\documentclass[letter,reqno,10pt]{amsart}
\usepackage{amssymb,amsmath,amsfonts}
\usepackage{mathrsfs}
\usepackage[left=1 in, right=1 in,top=1 in, bottom=1 in]{geometry}
\usepackage{mathenv}
\usepackage{framed}
\usepackage{subfigure}
\usepackage{float}
\usepackage{graphicx}

\newtheorem{theorem}{Theorem}[section]
\newtheorem{lemma}[theorem]{Lemma}

\newtheorem{remark}[theorem]{Remark}

\makeatletter
\renewcommand \theequation {%
\ifnum \c@section>\z@ \@arabic\c@section.%
\fi\@arabic\c@equation} \@addtoreset{equation}{section}
\makeatother

\providecommand{\ud}[1]{\mathrm{d}{#1}}
\providecommand{\abs}[1]{\left\vert#1\right\vert}
\providecommand{\nm}[1]{\left\Vert#1\right\Vert}

\providecommand{\tm}[2]{\left\Vert#1\right\Vert_{H^{#2}(\Omega)}}
\providecommand{\tms}[2]{\left\Vert#1\right\Vert_{H^{#2}(\Sigma)}}

\def\dt{\partial_t}
\def\p{\partial}
\def\ls{\lesssim}

\def\half{\frac{1}{2}}
\def\rt{\rightarrow}
\def\r{\mathbb{R}}
\def\no{\nonumber}
\def\ue{\mathrm{e}}
\def\ui{\mathrm{i}}
\def\ds{\displaystyle}

\def\z{\zeta}
\def\e{\eta}
\def\na{\nabla}
\def\de{\Delta}
\def\c{c}
\def\b{b}
\def\y{\ast}
\def\be{\bar\e}
\def\l{L_3}

\def\t{\mathbb{T}}
\def\s{\mathcal{S}}
\def\a{\mathcal{A}}
\def\dd{\mathbb{D}}

\def\en{\mathcal{E}}
\def\ben{\bar\en}
\def\di{\mathcal{D}}
\def\bdi{\bar\di}
\def\f{\mathcal{F}}

\def\nl{\nabla_{\a}}
\def\dl{\dd_{\a}}
\def\sl{\s_{\a}}
\def\ll{\de_{\a}}
\def\n{\mathcal{N}}
\def\al{\alpha}
\def\pp{\mathcal{P}}

\begin{document}
\title{Dynamics and Stability of Surface Waves with Bulk-Soluble Surfactants}

\author[I. Tice]{Ian Tice}
\address[I. Tice]{
   \newline\indent Department of Mathematical Sciences, Carnegie Mellon University
\newline\indent Pittsburgh, PA 15213, USA}
\email{iantice@andrew.cmu.edu}
\thanks{I. Tice was supported by a grant from the Simons Foundation (\#401468)}

\author[L. Wu]{Lei Wu}
\address[L. Wu]{
   \newline\indent Department of Mathematical Sciences, Carnegie Mellon University
\newline\indent Pittsburgh, PA 15213, USA}
\email{lwu2@andrew.cmu.edu}
\thanks{L. Wu was supported by NSF grant 0967140.}

\subjclass[2010]{Primary, 35Q30, 35R35, 76D45 ; Secondary, 35B40, 76E17, 76E99 }

\keywords{Surfactants, Navier-Stokes equations, interfacial stability}

\begin{abstract}
In this paper we study the dynamics of a layer of incompressible viscous fluid bounded below by a rigid boundary and above by a free boundary, in the presence of a uniform gravitational field.  We assume that a mass of surfactant is present both at the free surface and in the bulk of fluid, and that conversion from one species to the other is possible.  The surfactants couple to the fluid dynamics through the coefficient of surface tension, which depends on the the surface density of surfactants.  Gradients in this concentration give rise to Marangoni stress on the free surface.  In turn, the fluids advect the surfactants and distort their concentration through geometric distortions of the free surface. We model the surfactants in a way that allows absorption and desorption of surfactant between the surface and bulk.   We prove that small perturbations of the equilibrium solutions give rise to global-in-time solutions that decay to equilibrium at an exponential rate.  This establishes the asymptotic stability of the equilibrium solutions.
\end{abstract}

\maketitle

\pagestyle{myheadings} \thispagestyle{plain} \markboth{IAN TICE AND LEI WU}{DYNAMICS AND STABILITY OF SURFACE WAVE WITH SURFACTANT}

\section{Introduction}

\subsection{Model presentation: soluble surfactants}

Surfactants (a portmanteau for ``surface active agents'') are chemicals that can change the strength of surface tension when they collect at fluid free interfaces.  Variations in the surfactant concentration on the surface also give rise to tangential surface forces called Marangoni forces. The surfactant dynamics are driven by several effects: absorption and desorption from the free surface, fluid transport along the surface and in the bulk, and both bulk and surface diffusion.  We refer to the books  \cite{ed_bren_was,levich} and the review \cite{Sar} for a more thorough discussion of surfactant physics.  In manufacturing and industrial applications surfactants are a fundamental tool for  stabilizing bubble formation in processes such as foaming, emulsifying, and coating (see the books \cite{myers,rosen} for an exhaustive list of surfactant applications).  Surfactants also play a critical role in preventing the collapse of the lungs during breathing (see  \cite{hills} and the references therein) and are currently being developed as tools to aid in drug delivery in the lungs (see for example \cite{garoff_2, garoff}).

We consider a model of surfactants in which a single layer of fluid occupies the three-dimensional domain $\Omega(t)$ with free boundary surface $\Gamma(t)$.   In order to allow for surfactants concentrated on the free surface and in the bulk, we define the surface concentration $\hat{c}(\cdot,t): \Gamma(t) \to [0,\infty)$ as well as the bulk concentration $\hat{b}(\cdot,t) : \Omega(t) \to [0,\infty)$. The coefficient of surface tension on $\Gamma(t)$ depends on the surfactant concentration $\hat{c}$ via a relation $\sigma = \sigma(\hat{c})$, where we assume that the surface tension function satisfies:
\begin{equation}\label{sigma_assume}
\begin{cases}
\sigma \in C^3([0,\infty)) \\
\sigma \text{ is positive and strictly decreasing.}
\end{cases}
\end{equation}
The latter assumption comes from the fact that surfactants decrease the surface tension in higher concentration, and the former assumption is merely a technical assumption needed for our PDE analysis.

We will assume that the fluid is incompressible and viscous (with viscosity $\mu >0$) and that a uniform gravitational field $-g e_3 = (0,0,-g) \in \mathbb{R}^3$ is applied to the fluid (here $g>0$).  The fluid and surfactant dynamics then couple through the following system of equations (see \cite{ed_bren_was,levich} for derivations and precise definitions of the operators: we will soon reformulate these equations so do not fully define the operators here)
\begin{equation}\label{surf}
 \begin{cases}
\partial_t u + u \cdot \nabla u + \nabla p  = \mu \Delta u - g  e_3  \text{ and } \nabla \cdot u=0  & \text{in }$\Omega(t)$ \\
p\nu -\mu (\nabla u + \nabla u^T) \nu  =  - \sigma(\hat{c}) \mathcal{H}_{\Gamma(t)} \nu - \nabla_{\Gamma(t)} (\sigma(\hat{c})), & \text{on } $\Gamma(t)$ \\
 D_t \hat{c} + \hat{c} \nabla_{\Gamma(t)} \cdot u = \gamma \Delta_{\Gamma(t)} \hat{c} -\beta \nabla \hat{b} \cdot \nu &\text{on }$\Gamma(t)$ \\
 \dt \hat{b} + u \cdot \nabla \hat{b} = \beta \Delta \hat{b} & \text{in }$\Omega(t)$ \\
 \beta \nabla \hat{b} \cdot \nu = \omega(\hat{c},\hat{b}) &\text{on }$\Gamma(t)$.
\end{cases}
\end{equation}
Here $\nabla_{\Gamma(t)}$ denotes the surface gradient on $\Gamma(t)$,  $D_t$ is a temporal derivative along the flowing surface, $\nabla_{\Gamma(t)}\cdot$ is the surface divergence, $\Delta_{\Gamma(t)}$ is the surface Laplacian,  $\gamma > 0$ is the surface surfactant diffusion constant, $\beta > 0$ is the bulk surfactant diffusion constant, and $\mathcal{H}_{\Gamma(t)}$ is twice the mean-curvature operator on $\Gamma(t)$.  The first two equations in \eqref{surf} are the usual incompressible Navier-Stokes equations.  The third equation is the balance of stress on the free surface, and the right-hand side shows that two stresses are generated by the surfactants.  The first term is a normal stress caused by surface curvature, and the second is a tangential stress, known as the Marangoni stress, caused by gradients in the surfactant concentration on the surface.   The fourth equation in \eqref{surf} shows that the surface surfactant concentration changes due to flow on the surface as well as diffusion and the normal derivative of $\hat{b}$, and the fifth equation is a advection-diffusion equation for the bulk surfactant concentration.

The sixth equation in \eqref{surf} is of fundamental importance in surfactant modeling, as it gives the law for the conversion of bulk surfactant into surface surfactant  at the free interface.  We refer, for instance, to the book \cite{ed_bren_was} for a derivation of such a law from physical kinetic arguments.  More recent work \cite{diamant_andelman,diamant_ariel_etal,garcke_lam_stinner} has derived laws of a similar form for two-phase flows using free-energy arguments.   Here the function $\omega: [0,\infty) \times [0,\infty) \to \mathbb{R}$ describes the surfactant flow from the surface to bulk phase.  We will assume that $\omega$ obeys the following:
\begin{equation}\label{omega_assume}
\begin{cases}
 \omega \in C^4([0,\infty \times [0,\infty)) \\
 \text{there exists }f \in C^2([0,\infty);[0,\infty)) \text{ such that } \omega(c,b) >0 \Leftrightarrow c > f(b) \\
 f'(b) >0 \text{ for all } b \ge 0.
\end{cases}
\end{equation}
In the parlance of surface chemistry, the set $\{\omega(c,b) =0\}$ is called an isotherm.  It represents the values of surface and bulk surfactant that do not result in conversion between species, and consequently any equilibrium configuration must lie on the isotherm.  Our first assumption on $\omega$ is purely technical, but the second and third indicate that the isotherm is given as the graph of the function $f$. The sign condition on $f'$ guarantees that more surface surfactant is needed to form equilibrium when more bulk surfactant is present in the fluid.  As an example, in the Langmuir absorption model (see for instance \cite{ed_bren_was}) we have
\begin{equation}
 \omega(\c,\b)=-\b(k_1-\c)+k_2\c \text{ for physical constants }k_1,k_2>0.
\end{equation}
In this case $f(b) = bk_1/(k_2 + b)$.

Surfactants have been extensively studied in the physics literature, and we will not attempt to survey that literature here.  In contrast, they have not received extensive attention in the mathematics literature.  Kwan-Park-Shen \cite{kwan_park_shen} and Xu-Li-Lowengrub-Zhao \cite{xu_li_low_zhao} developed numerical studies of surfactant dynamics.  The local well-posedness of a two-phase bubble model without gravity was proved by Bothe-Pr\"{u}ss-Simonett \cite{bo_pr_si_1} in the context of diffusion-limited absorption, which means that the last condition in \eqref{surf} is replaced by the condition $c = f(b)$ so that the surfactant phases are in constant equilibrium.    The linear stability of the same model was studied by Bothe-Pr\"{u}ss \cite{bo_pr_1}.  In \cite{Kim.Tice2016} Kim-Tice studied the dynamics of surfactants without absorption and proved the existence of global-in-time solutions near equilibrium as well as their asymptotic stability.  Diffuse interface models with absorption laws similar to \eqref{surf} have also recently been studied by Garcke-Lam-Stinner \cite{garcke_lam_stinner} and Abels-Garcke-Lam-Weber \cite{abels_etal}.

The principal goal of the present paper is to extend the techniques developed in \cite{Kim.Tice2016} to handle the absorption model presented in \eqref{surf}.  We will construct global solutions near equilibrium and prove their long-time decay to equilibrium.

\subsection{Problem presentation: formulation of equations}

We now give a precise description of the equations of motion.  To begin we assume that the fluid occupies a three-dimensional domain $\Omega(t)$ that is horizontally periodic with a flat rigid bottom and free upper surface:
\begin{eqnarray}
\Omega(t)=\{y\in\Sigma\times\r: -\l < y_3< \e(t,y_1,y_2)\},
\end{eqnarray}
where the horizontal cross-section is given by
\begin{eqnarray}
\Sigma=(L_1\t)\times(L_2\t) = (\mathbb{R} / L_1 \mathbb{Z}) \times (\mathbb{R} / L_2 \mathbb{Z})
\end{eqnarray}
for $L_1,L_2>0$ fixed periodicity lengths and depth $\l>0$, and the free surface function is $\eta : \Sigma \times [0,\infty) \to (-\l,\infty)$.  We will denote the fixed lower boundary by
\begin{eqnarray}
\Sigma_b=\{y\in\Sigma\times\r: y_3=-\l\}.
\end{eqnarray}
and the moving upper boundary by
\begin{eqnarray}
\Gamma(t)=\{y\in\Sigma\times\r: y_3 = \e(t,y_1,y_2)\},
\end{eqnarray}
which is the graph of the unknown function $\e$.

For each $t\ge 0$, the dynamics of the fluid and surfactants are described by the following unknowns.  The fluid velocity and pressure are  $u(\cdot,t): \Omega(t) \to \mathbb{R}^3$ and $p(\cdot,t) : \Omega(t) \to \mathbb{R}^3$, and the free surface function is $\e(\cdot,t): \Sigma\rt\r$.   It will be convenient to study the projection onto $\Sigma$ of the surface surfactant concentration, so we define  $\tilde\c(\cdot,t):\Sigma\rt\r$ via $\tilde\c(y_{\y},t)=\hat\c\Big(y_{\y},\e(y_{\y},t),t\Big)$ where $\hat{c}(\cdot,t)$ is defined on the moving surface $\Gamma(t)$ and $y_{\y}=(y_1,y_2) \in \Sigma$.  To be consistent with the new notation we also write $\tilde{b}(\cdot,t) = \hat{b}(\cdot,t) : \Omega(t) \to [0,\infty)$ for the bulk surfactant concentration.

In the following, we will employ the horizontal differential operators
\begin{equation}
\na_{\y}F = \sum_{i=1}^2\Big(\p_iF e_i\Big) \text{ and }
\na_{\y}\cdot  G = \sum_{i=1}^2\Big(\p_iG_i\Big)
\end{equation}
acting on scalar and vector fields, respectively.  Similarly, we define $\nu_{\y}=(\nu_1,\nu_2)$  to be the horizontal components of outward unit normal vector on the moving surface:
\begin{eqnarray}
\nu=(\nu_1,\nu_2,\nu_3)=\frac{(-\na_{\y}\e,1)}{\sqrt{1+\abs{\na_{\y}\e}^2}}.
\end{eqnarray}
The surface differential operators are defined as
\begin{eqnarray}
\p_{\Gamma,1}&=&\frac{1+\abs{\p_2\e}^2}{1+\abs{\na\e}^2}\p_1-\frac{\p_1\e\p_2\e}{1+\abs{\na\e}^2}\p_2,\\
\p_{\Gamma,2}&=&\frac{1+\abs{\p_1\e}^2}{1+\abs{\na\e}^2}\p_2-\frac{\p_1\e\p_2\e}{1+\abs{\na\e}^2}\p_1,\\
\p_{\Gamma,3}&=&\frac{\p_1\e}{1+\abs{\na\e}^2}\p_1+\frac{\p_2\e}{1+\abs{\na\e}^2}\p_2.
\end{eqnarray}
Then we have for $f$ a scalar and $g$ a vector field defined in $\Sigma$,
\begin{equation}
\na_{\Gamma}f = \sum_{i=1}^3\Big(\p_{\Gamma,i}f  e_i\Big) \text{ and }
\na_{\Gamma}\cdot  g = \sum_{i=1}^3\Big(\p_{\Gamma,i}g_i\Big).
\end{equation}
Also, for $F$ and $G$ defined in $\Gamma(t)$,
\begin{equation}
\na_{\Gamma}F = \na_{\Gamma}(F\circ\e) \text{ and }
\na_{\Gamma}\cdot  G = \na_{\Gamma}\cdot(G\circ\e).
\end{equation}
Thus, we define $\de_{\Gamma}=\na_{\Gamma}\cdot\na_{\Gamma}$.

The fluid, surface,  and surfactant unknowns $(u,p,\e,\tilde\c,\tilde\b)$ must satisfy the following system of equations:
\begin{equation}\label{origin}
\left\{
\begin{array}{ll}
\dt u+u\cdot\na u+\na p=\mu\de u&\ \ \text{in}\ \ \Omega(t),\\
\na\cdot u=0 &\ \ \text{in}\ \ \Omega(t),\\
(pI-\mu\dd u)\nu=g\e\nu-\sigma(\tilde\c)H(\e)\nu-\na_{\Gamma}\sigma(\tilde\c)&\ \ \text{on}\ \ \Gamma(t),\\
u=0&\ \ \text{on}\ \ \Sigma_b,\\
\dt\e=u_3-u_1\p_{y_1}\e-u_2\p_{y_2}\e &\ \ \text{on}\ \ \Gamma(t),\\
\ \\
\dt\tilde\c+u\cdot\na_{\y}\tilde\c+\tilde\c\na_{\Gamma}\cdot u=\gamma\de_{\Gamma}\tilde\c-\beta\na\tilde\b\cdot\nu &\ \ \text{on}\ \ \Gamma(t),\\
\dt\tilde\b+u\cdot\na \tilde\b=\beta\de\tilde\b&\ \ \text{in}\ \ \Omega(t),\\
\beta\na\tilde\b\cdot\nu=\omega(\tilde\c,\tilde\b)&\ \ \text{on}\ \ \Gamma(t),\\
\na\tilde\b\cdot\nu=0&\ \ \text{on}\ \ \Sigma_b,\\
\ \\
u(t=0)=u_0&\ \ \text{in}\ \ \Omega_0,\\
\e(t=0)=\e_0&\ \ \text{on}\ \ \Sigma,\\
\tilde\b(t=0)=\tilde\b_0&\ \ \text{in}\ \ \Omega_0,\\
\tilde\c(t=0)=\tilde\c_0&\ \ \text{on}\ \ \Sigma.
\end{array}
\right.
\end{equation}
The first set of equations describes the motion of fluid under the influence of surfactant and the second set describes the absorption and desorption of surfactant with convection and diffusion both on the surface and in the bulk. Here viscosity $\mu$, gravity $g$, convection strength $\beta$, and diffusion strength $\gamma$ are positive constants, $I$ the $3\times3$ identity matrix, $(\dd u)_{ij}=\p_ju_i+\p_iu_j$ the symmetric gradient, and the mean-curvature operator is
\begin{eqnarray}
H(\e)=\na_{\y}\cdot\left(\frac{\na_{\y}\e}{\sqrt{1+\abs{\na_{\y}\e}^2}}\right).
\end{eqnarray}

Employing a standard scaling argument in space and time we can eliminate two of the physical constants at the expense of possibly renaming the rest.  We will do this in order to set  $\mu=g=1$.  We will employ this convention throughout the rest of this paper.  Also, we assume the initial surface $\e_0$ satisfies the zero-average condition
\begin{eqnarray}\label{z_avg}
\int_{\Sigma}\e_0=0,
\end{eqnarray}
and the total mass of surfactant
\begin{eqnarray}\label{conserv_c}
\int_{\Sigma}\tilde\c_0\sqrt{1+\abs{\na_{\y}\e_0}^2}+\int_{\Omega(t)}\tilde\b_0 :=M >0.
\end{eqnarray}
Standard calculations, which we omit here for the sake of brevity, show that these conditions persist in time, i.e.
\begin{equation}\label{mass}
 \int_{\Sigma} \e(\cdot,t) = 0 \text{ and }\int_{\Sigma}\tilde\c(\cdot,t) \sqrt{1+\abs{\na_{\y}\e(\cdot,t)}^2}+\int_{\Omega(t)}\tilde\b_0(\cdot,t) =M
\end{equation}
for $t>0$.  Note, though, that the latter condition only guarantees that the total surfactant mass is preserved in time.  The portion of mass in each phase (bulk and surface) does not have to be preserved.

\subsection{Energy-dissipation structure}

We may follow the computations in  \cite{Kim.Tice2016}, with the help of Lemma \ref{app lemma 2}, to derive the following  energy-dissipation structure from the fluid equations:
\begin{eqnarray}\label{it 01}
\frac{\ud{}}{\ud{t}}\bigg(\half\int_{\Omega(t)}\abs{u}^2+\half\int_{\Sigma}\abs{\e}^2\bigg)
+\half\int_{\Omega(t)}\abs{\dd u}^2=-\int_{\Sigma}\sigma(\tilde\c)(\na_{\Gamma}\cdot u)\sqrt{1+\abs{\na_{\y}\e}^2}.
\end{eqnarray}
We might hope to use similar computations with $\tilde{c}$ and $\tilde{b}$ to produce further energy-dissipation equations that sum with the above to produce an equation with no exchange terms on the right.  Unfortunately, this does not work if we directly work with the $\tilde{c}$ and $\tilde{b}$.  Instead we extend an idea used in  \cite{bo_pr_1,Kim.Tice2016} for the case without absorption, which shifts from an $L^2-$based estimate to a more complicated estimate.  To this end, for any $r\in(0,\infty)$ we define the auxiliary functions $\zeta_r,\phi_r :[0,\infty) \to \mathbb{R}$ via
\begin{eqnarray}
\z_r(s)&=&  s \left( \frac{\sigma(r)}{r} - \int_{r}^s \frac{\sigma(z)}{z^2} \ud{z}  \right)  = \sigma(s)-s\int_r^{s}\frac{\sigma'(z)}{z}\ud{z},\\
\phi_r(s)&=&C_0+\int_{0}^s\frac{z\sigma'(f(z))f'(z)}{f(z)}\ud{z}-s\int_r^{f(s)}\frac{\sigma'(z)}{z}\ud{z}
\end{eqnarray}
for a constant $C_0>0$ chosen such that $\phi_r(s) >0$ for all $s \ge 0$.  This is possible since $f'(0) >0$ implies that
\begin{equation}
\lim_{z\to0} \frac{zf'(z)}{f(z)} =
\begin{cases}
 0 & \text{if } $f(0)>0$ \\
 1 & \text{if } $f(0)=0$.
\end{cases}
\end{equation}
Note also that
\begin{equation}
 0 \le s \le r \Rightarrow \abs{s\int_r^{s}\frac{\sigma'(z)}{z}\ud{z}} \le  \Vert \sigma'\Vert_{L^\infty([0,r])} s\log\left( \frac{r}{s} \right),
\end{equation}
so the $\sigma'$ integral in the definition of $\z_r$ is well-defined for all $s \ge 0$.  A similar argument shows that the second integral in the definition of $\phi_r$ is well-defined even when $f(0)=0$.

By construction, we have
\begin{equation}
\z_r'(s)=-\int_r^{s}\frac{\sigma'(z)}{z}\ud{z} \text{ and }
\phi_r'(s)=-\int_r^{f(s)}\frac{\sigma'(z)}{z}\ud{z},
\end{equation}
as well as
\begin{eqnarray}
\z_r(s),\phi_r(s),\z_r''(s),\phi_r''(s)&>&0,\\
\z_r(s)-s\z_r'(s)&=&\sigma(s).
\end{eqnarray}
Moreover, $\z_r$ attains its minimum at $s=r$ and $\phi_r$ attains its minimum if there exists $s \ge 0$ such that $f(s) =r$.  With these auxiliary functions in hand, we may then follow \cite{Kim.Tice2016}, using the identities in Lemma \ref{app lemma 3}, to see that
\begin{eqnarray}\label{it 02}
\frac{\ud{}}{\ud{t}}\int_{\Sigma}\zeta_r(\tilde\c)\sqrt{1+\abs{\na_{\y}\e}^2}+\gamma\int_{\Sigma}\zeta_r''(\tilde\c)\abs{\na_{\Gamma}\tilde\c}^2\sqrt{1+\abs{\na_{\y}\e}^2}\\
+\int_{\Sigma}\zeta'_r(\tilde\c)\omega(\tilde\c,\tilde\b)\sqrt{1+\abs{\na_{\y}\e}^2}&=&\int_{\Sigma}\sigma(\tilde\c)(\na_{\Gamma}\cdot u)\sqrt{1+\abs{\na_{\y}\e}^2},\no
\end{eqnarray}
and
\begin{eqnarray}\label{it 03}
\frac{\ud{}}{\ud{t}}\int_{\Omega(t)}\phi_r(\tilde\b)+\beta\int_{\Omega(t)}\phi''_r(\tilde\b)\abs{\na\tilde\b}^2
-\int_{\Sigma}\phi'_r(\tilde\b)\omega(\tilde\c,\tilde\b)\sqrt{1+\abs{\na_{\y}\e}^2}&=&0.
\end{eqnarray}
Summing up \eqref{it 01}, \eqref{it 02} and \eqref{it 03}, we obtain the full energy-dissipation structure:
\begin{eqnarray}
\frac{\ud{}}{\ud{t}}\bigg(\half\int_{\Omega(t)}\abs{u}^2+\half\int_{\Sigma}\abs{\e}^2
+\int_{\Omega(t)}\phi_r(\tilde\b)+\int_{\Sigma}\zeta_r(\tilde\c)\sqrt{1+\abs{\na_{\y}\e}^2}\bigg)\\
+\bigg(\half\int_{\Omega(t)}\abs{\dd u}^2+\beta\int_{\Omega(t)}\phi''_r(\tilde\b)\abs{\na\tilde\b}^2+\gamma\int_{\Sigma}\zeta_r''(\tilde\c)\abs{\na_{\Gamma}\tilde\c}^2\sqrt{1+\abs{\na_{\y}\e}^2}\bigg)\no\\
+\int_{\Sigma}\Big(\zeta'_r(\tilde\c)-\phi'_r(\tilde\b)\Big)\omega(\tilde\c,\tilde\b)\sqrt{1+\abs{\na_{\y}\e}^2}&=&0.\no
\end{eqnarray}

The terms appearing on the first and second lines of the energy-dissipation equality are clearly positive semi-definite, but at first glance it is not clear that the term on the third line possesses a good sign.  To show that this is in fact true we first note that the definitions of $\zeta_r$ and $\phi_r$ allow us to compute
\begin{eqnarray}
\Big(\zeta'_r(\tilde\c)-\phi'_r(\tilde\b)\Big)\omega(\tilde\c,\tilde\b)=\bigg(\int_{\tilde{\c}}^{f(\tilde\b)}\frac{\sigma'(z)}{z}\ud{z}\bigg)\omega(\tilde\c,\tilde\b).
\end{eqnarray}
Employing the assumptions on $\sigma$ \eqref{sigma_assume} and $\omega$ \eqref{omega_assume}, we then find that
\begin{eqnarray}
\left\{
\begin{array}{l}
\tilde\c>f(\tilde\b)\Rightarrow \ds\int_{\tilde{\c}}^{f(\tilde\b)}\frac{\sigma'(z)}{z}\ud{z}>0\ \ \text{and}\ \ \omega(\tilde\c,\tilde\b)>0,\\
\tilde\c<f(\tilde\b)\Rightarrow \ds\int_{\tilde{\c}}^{f(\tilde\b)}\frac{\sigma'(z)}{z}\ud{z}<0\ \ \text{and}\ \ \omega(\tilde\c,\tilde\b)<0.
\end{array}
\right.
\end{eqnarray}
Thus, in both cases we have
\begin{eqnarray}
\int_{\Sigma}\Big(\zeta'_r(\tilde\c)-\phi'_r(\tilde\b)\Big)\omega(\tilde\c,\tilde\b)\sqrt{1+\abs{\na\e}^2}\geq0.
\end{eqnarray}
Therefore, if we define
\begin{eqnarray}
E&=&\half\int_{\Omega(t)}\abs{u}^2+\half\int_{\Sigma}\abs{\e}^2+\int_{\Omega(t)}\phi_r(\tilde\b)+\int_{\Sigma}\zeta_r(\tilde\c)\sqrt{1+\abs{\na\e}^2},  \\
D&=&\half\int_{\Omega(t)}\abs{\dd u}^2+\beta\int_{\Omega(t)}\phi''_r(\tilde\b)\abs{\na\tilde\b}^2+\gamma\int_{\Sigma}\zeta_r''(\tilde\c)\abs{\na_{\Gamma}\tilde\c}^2\sqrt{1+\abs{\na\e}^2} \\
&+&   \int_{\Sigma}\Big(\zeta'_r(\tilde\c)-\phi'_r(\tilde\b)\Big)\omega(\tilde\c,\tilde\b)\sqrt{1+\abs{\na\e}^2} \nonumber,
\end{eqnarray}
then we know $E$ and $D$ are positive semi-definite and satisfy the energy-dissipation equation
\begin{eqnarray}\label{basic_ed}
\frac{\ud{E}}{\ud{t}}+D = 0.
\end{eqnarray}

\subsection{Equilibria}

We now turn our attention to a discussion of the equilibrium solutions.  Assume we have a time-independent solution.  This assumption and \eqref{basic_ed} imply that
\begin{equation}
\abs{\dd u}^2 = \phi''_r(\tilde\b)\abs{\na\tilde\b}^2 = \zeta_r''(\tilde\c)\abs{\na_{\Gamma}\tilde\c}^2 = \Big(\zeta'_r(\tilde\c)-\phi'_r(\tilde\b)\Big) \omega(\tilde{c},\tilde{b}) = 0.
\end{equation}
Since $u=0$ on $\Sigma_b$, we know from Korn's inequality that $u=0$. Considering that $\phi_r$ and $\zeta_r$ are strictly convex, we know $\na\tilde\b=\na_{\Gamma}\tilde\c=0$. Hence, we have $\tilde\b=\b_0\in(0,\infty)$, $\tilde\c=\c_0\in(0,\infty)$. Therefore, by the boundary condition of fluid equation, we know $p=\e=0$.   Furthermore,  $\b_0$ and $\c_0$ must preserve the total mass as in \eqref{mass},
\begin{equation}
\abs{\Sigma}_2 c_0 + \abs{\Omega}_3 b_0 = \int_\Sigma c_0 +   \int_\Omega b_0 =M,
\end{equation}
where $\abs{\cdot}_k$ denotes $k-$dimensional Lebesgue measure,  as well as satisfy the condition $\omega(\c_0,\b_0)=0$.  The latter requires that $f(\b_0) = \c_0$.  Since $f'>0$, for given total mass $M > \abs{\Sigma}_2 f(0)$, we have a unique solution $\b_0,\c_0 = f(\b_0) >0$.

In summary, we have the equilibrium configuration
\begin{equation}\label{equilibrium_config}
\left\{
\begin{array}{rcl}
u&=&0\\
p&=&0\\
\e&=&0\\
\omega(\c_0,\b_0)&=&0,\\
\abs{\Sigma}_2 \c_0+\abs{\Omega}_3\b_0&=&M,
\end{array}
\right.
\end{equation}
where $c_0,b_0 \in (0,\infty)$ are uniquely determined by the choice of mass $0 \le \abs{\Sigma}_2 f(0) < M <\infty$ and the form of the surfactant flux function $\omega$.

It is worth noting that in the definition of $\zeta_r$ and $\phi_r$ above we have that $\zeta_r$ obtains its minimal value at $s=r$ and $\phi_r$ obtains its minimum at $s$ such that $f(s) =r$.  If we choose $r = c_0$ then the minimal value of $\zeta_r$ is obtained at $c_0$ and the minimal value of $\phi_r$ is obtained at $f(c_0) = b_0$.  Thus \eqref{basic_ed} suggests that the equilibrium configuration will be dynamically stable.

We conclude our discussion of the equilibria by introducing some notation that will be useful later.  We linearize $\omega$ around $(\c_0,\b_0)$ as
\begin{eqnarray}\label{linearize}
\omega(\c_0+h,\b_0+d)&=&\omega(\c_0,\b_0)+\omega_{\c0}h+\omega_{\b0}d+2hd\int_0^1(1-s)\p_{\c\b}\omega(\c_0+sh,\b_0+sd)\ud{s}\\
&&h^2\int_0^1(1-s)\p_{\c\c}\omega(\c_0+sh,\b_0+sd)\ud{s}+d^2\int_0^1(1-s)\p_{\b\b}\omega(\c_0+sh,\b_0+sd)\ud{s}\no\\
&=&\omega(\c_0,\b_0)+\omega_{\c0}h+\omega_{\b0}d+O(\abs{h}^2+\abs{d}^2)=\omega_{\c0}h+\omega_{\b0}d+O(\abs{h}^2+\abs{d}^2),\no
\end{eqnarray}
where we have written
\begin{equation}
 \omega_{\c0} := \partial_c \omega(c_0,b_0) \text{ and }  \omega_{\b0} := \partial_b \omega(c_0,b_0).
\end{equation}
Since $\omega(f(b),b)=0$ for all $b \ge 0$, we may differentiate to see that
\begin{eqnarray}
\omega_{\c0}f'(\b_0)+\omega_{\b0}=0,
\end{eqnarray}
which implies
\begin{eqnarray}
\frac{\omega_{\b0}}{\omega_{\c0}}=-f'(\b_0).
\end{eqnarray}
Also, since $\omega(\c,\b)>0$ leads to $\c> f(\b)$, we have $\omega_{\c0}>0$ and $\omega_{\b0}<0$.

\subsection{Reformulation}

In order to work in a fixed domain, we employ a frequently used transformation: see \cite{B,GT1,GT2,Kim.Tice2016,WTK}.  We define the equilibrium fluid domain $\Omega$ via
\begin{eqnarray}
\Omega=\{y\in\Sigma\times\r: -\l < y_3 < 0\},
\end{eqnarray}
which possesses $\Sigma$ as upper boundary and $\Sigma_b$ as lower boundary. We can then view  $\e$ as defined on $\Sigma$. Assume $\be$ is the harmonic extension of $\e$ into $\Sigma \times (-\infty,0)$ defined as in \eqref{Poisson}. Define mapping
\begin{eqnarray}\label{phi_def}
\Phi: \Omega\ni x=(x_1,x_2,x_3) \mapsto \bigg(x_1,x_2,x_3 + \be(x,t) \left(1+\frac{x_3}{\l}\right)\bigg)\in\Omega(t).
\end{eqnarray}
We may easily check that $\Phi$ maps $\Sigma$ to $\Gamma(t)$ and maps $\Sigma_b$ to $\Sigma_b$. The Jacobian matrix is
\begin{eqnarray}
\na\Phi=\left(
\begin{array}{ccc}
1&0&0\\
0&1&0\\
A&B&J
\end{array}
\right),
\end{eqnarray}
and the transform matrix is
\begin{eqnarray}
\a=(\na\Phi)^{-T}=\left(
\begin{array}{ccc}
1&0&-AK\\
0&1&-BK\\
0&0&K
\end{array}
\right),
\end{eqnarray}
where
\begin{eqnarray}
&&A=\tilde{L}\frac{\p\be}{\p x_1} ,\ \ B=\tilde{L}\frac{\p\be}{\p x_2},\\
&&J=1+\frac{\be}{\l}+\tilde L\frac{\p\be}{\p x_3},\ \ K=\frac{1}{J},\  \tilde L=1+\frac{x_3}{\l}.\no
\end{eqnarray}
In this new coordinate system, \eqref{origin} is transformed into
\begin{equation}\label{transform}
\left\{
\begin{array}{ll}
\dt u-\dt\be\tilde LK\p_3 u+u\cdot\nl u-\ll u+\nl p=0&\ \ \text{in}\ \ \Omega,\\
\nl\cdot u=0 &\ \ \text{in}\ \ \Omega,\\
(pI-\dl u)\n=\e\n-\sigma(\tilde\c)H(\e)\n-\sqrt{1+\abs{\na_{\y}\e}^2}\sigma'(\tilde\c)\na_{\Gamma}\tilde\c&\ \ \text{on}\ \ \Sigma,\\
u=0&\ \ \text{on}\ \ \Sigma_b,\\
\dt\e=u\cdot\n &\ \ \text{on}\ \ \Sigma,\\
\ \\
\dt\tilde\c+u\cdot\na_{\y}\tilde\c+\tilde\c\na_{\Gamma}\cdot u=\gamma\de_{\Gamma}\tilde\c-\beta\nl\tilde\b\cdot\dfrac{\n}{\sqrt{1+\abs{\na_{\y}\e}^2}} &\ \ \text{on}\ \ \Sigma,\\
\dt\tilde\b-\dt\be\tilde LK\p_3\tilde\b+u\cdot\nl \tilde\b=\beta\ll\tilde\b&\ \ \text{in}\ \ \Omega,\\
\nl\tilde\b\cdot\dfrac{\n}{\sqrt{1+\abs{\na_{\y}\e}^2}}=\omega(\tilde\c,\tilde\b)&\ \ \text{on}\ \ \Sigma,\\
\p_3\tilde\b=0&\ \ \text{on}\ \ \Sigma_b.
\end{array}
\right.
\end{equation}
Here we have defined the transformed operators as follows:
\begin{equation}
\begin{array}{l}
(\nl f)_i=\a_{ij}\p_jf,\\
\nl\cdot g=\a_{ij}\p_jg_i,\\
\ll f=\nl\cdot\nl f,\\
\n=(-\p_1\e,-\p_2\e,1),\\
(\dl u)_{ij}=\a_{ik}\p_ku_j+\a_{jk}\p_ku_i,\\
\sl(p,u)=pI-\dl u,
\end{array}
\end{equation}
where the summation should be understood in the Einstein convention.
If we extend the divergence $\nl\cdot$ to act on symmetric tensors in the
natural way,
then a straightforward computation reveals $\nl\cdot\sl(p,u)=\nl p-\ll u$ for vector fields satisfying $\nl\cdot u=0$.

\subsection{Perturbation}

We will study the case when the initial data is a perturbation of the equilibrium. Define the perturbation as $\c=\tilde\c-\c_0$ and $\b=\tilde\b-\b_0$, then $(u,p,\e,\c,\b)$ satisfies
\begin{equation}\label{perturbation}
\left\{
\begin{array}{ll}
\dt u-\dt\be\tilde LK\p_3 u+u\cdot\nl u-\ll u+\nl p=0&\ \ \text{in}\ \ \Omega,\\
\nl\cdot u=0 &\ \ \text{in}\ \ \Omega,\\
(pI-\dl u)\n=\e\n-\sigma(\c+\c_0)H(\e)\n-\sqrt{1+\abs{\na_{\y}\e}^2}\sigma'(\c+\c_0)\na_{\Gamma}\c&\ \ \text{on}\ \ \Sigma,\\
u=0&\ \ \text{on}\ \ \Sigma_b,\\
\dt\e=u\cdot\n &\ \ \text{on}\ \ \Sigma,\\
\ \\
\dt\c+u\cdot\na_{\y}\c+(\c+\c_0)\na_{\Gamma}\cdot u=\gamma\de_{\Gamma}\c-\beta\nl\b\cdot\dfrac{\n}{\sqrt{1+\abs{\na_{\y}\e}^2}} &\ \ \text{on}\ \ \Sigma,\\
\dt\b-\dt\be\tilde LK\p_3\b+u\cdot\nl \b=\beta\ll\b&\ \ \text{in}\ \ \Omega,\\
\nl\b\cdot\dfrac{\n}{\sqrt{1+\abs{\na_{\y}\e}^2}}=\omega(\c+\c_0,\b+\b_0)&\ \ \text{on}\ \ \Sigma,\\
\p_3\b=0&\ \ \text{on}\ \ \Sigma_b.
\end{array}
\right.
\end{equation}
In the following, we will write $\sigma_0=\sigma(\c_0)$ and $\sigma'_0=\sigma'(\c_0)$.

\section{Main results and discussion}

\subsection{Main result}

To properly state the main results we must first define energy and dissipation functionals that will be used throughout the paper.   We define the full energy to be
\begin{eqnarray}\label{def_E}
\en&=&\tm{u}{2}^2+\tm{\dt u}{0}^2+\tm{p}{1}^2+\tms{\e}{3}^2+\tms{\dt\e}{\frac{3}{2}}^2+\tms{\dt^2\e}{-\frac{1}{2}}^2\\
&&+\tm{\b}{2}^2+\tm{\dt\b}{0}^2+\tms{\c}{2}^2+\tms{\dt\c}{0}^2,\no
\end{eqnarray}
and full dissipation to be
\begin{eqnarray}\label{def_D}
\di&=&\tm{u}{3}^2+\tm{\dt u}{1}^2+\tm{p}{2}^2+\tms{\e}{\frac{7}{2}}^2+\tms{\dt\e}{\frac{5}{2}}^2+\tms{\dt^2\e}{\frac{1}{2}}^2\\
&&+\tm{\b}{3}^2+\tm{\dt\b}{1}^2+\tms{\c}{3}^2+\tms{\dt\c}{1}^2.\no
\end{eqnarray}
Here the spaces $H^s$ denote the usual $L^2-$based Sobolev spaces of order $s$.

Now we present the main theorem on a priori estimates for solutions to \eqref{perturbation}.

\begin{theorem}[Proved later in Theorem \ref{apriori_est}]\label{aprioris_intro}
Suppose that $(u,p,\eta,c,b)$ solves \eqref{perturbation} on the temporal interval $[0,T]$.  Let $\en$ and $\di$ be as defined in \eqref{def_E} and \eqref{def_D}. Then there exists a universal constant $0 < \delta$ (independent of $T$) such that if
\begin{equation}
 \sup_{0\le t \le T} \en(t) \le \delta \text{ and } \int_0^T \di(t) dt < \infty,
\end{equation}
then
\begin{equation}
\sup_{0\le t \le T} e^{\lambda t} \en(t) + \int_0^T \di(t)dt \ls \en(0)
\end{equation}
for all $t \in [0,T]$, where $\lambda >0$ is a universal constant.
\end{theorem}

This theorem tells us that if solutions exist for which the energy functional $\mathcal{E}$ remains small and the dissipation functional is integrable in time, then in fact we have much stronger information: the energy decays exponentially and the integral of the dissipation is controlled by the initial energy.  In order for this result to be useful we must couple it with a local existence result.  By now it is well-understood how to construct local-in-time solutions for solutions to problems of the form \eqref{perturbation} once the corresponding a priori estimates are understood: we refer for instance to \cite{GT1,WTK,Wu} for local existence results in spaces determined by energies and dissipations of the form \eqref{def_E} and \eqref{def_D}, and to \cite{bo_pr_si_1} for results that employ $L^p-$maximal regularity techniques.  Consequently, in the interest of brevity, we will not attempt to prove a local existence result in the present paper.  Instead we will simply state the result that one can prove by modifying the known methods in straightforward ways.

Given the initial data $u_0,\eta_0,\tilde{c}_0$,$\tilde{b}_0$, we need to construct the initial data $\partial_t u(\cdot,0)$, $\partial_t \eta(\cdot,0)$, $\dt c(\cdot,0)$, $\dt b(\cdot,0)$, and $p(\cdot,0)$.  To construct these we  require a compatibility condition for the data.  To state this properly we  define the orthogonal projection onto the tangent space of the surface $\Gamma(0) = \{x_3 = \eta_0(x_\ast)\}$ according to
\begin{equation}
 \Pi_0 v = v - (v\cdot \mathcal{N}_0) \mathcal{N}_0 \abs{\mathcal{N}_0}^{-2}
\end{equation}
for $\mathcal{N}_0 = (-\p_1 \eta_0,-\p_2 \eta_0,1)$.  Then the compatibility conditions for the data read
\begin{equation}\label{compat_cond}
\begin{cases}
\Pi_0 (\mathbb{D}_{\mathcal{A}_0} u_0 \mathcal{N}_0) - \sqrt{1+\abs{\nabla_\ast \eta_0}^2} \sigma'(\tilde{c}_0) \nabla_{\Gamma_0} \tilde{c}_0 =0 &\text{on }$\Sigma$ \\
\nabla_{\mathcal{A}_0}\cdot u_0 =0 & \text{in }$\Omega$ \\
u_0 =0 & \text{on } $\Sigma_b$,
\end{cases}
\end{equation}
where here $\mathcal{A}_0$ and $\Gamma_0$ are determined by $\eta_0$.  To state the local result we will also need to define  $\mathcal{H}^1 := \{ u \in H^1(\Omega) \;\vert\;  u\vert_{\Sigma_b}=0\}$ and
\begin{equation}\label{X_space_def}
\mathcal{X}_T = \{u \in L^2([0,T];\mathcal{H}^1)\;\vert\; \nabla_{\mathcal{A}(t)} \cdot u(t)=0 \text{ for a.e. } t\}.
\end{equation}

Having stated the compatibility conditions, we can now state the local existence result.

\begin{theorem}\label{lwp_intro}
Let $u_0 \in H^2(\Omega)$,   $\eta_0 \in H^3(\Sigma)$, $\tilde{c}_0 \in H^2(\Sigma)$, $\tilde{b}_0 \in H^2(\Omega)$, and assume that $\eta_0$ satisfies \eqref{z_avg} and $(\tilde{c}_0,\tilde{b}_0)$  satisfy \eqref{conserv_c}, where $M > \abs{\Sigma}_2 f(b_0)$ is the equilibrium mass associated to the equilibrium surfactant concentration $(c_0,b_0)$ given in \eqref{equilibrium_config}.   Further assume that the initial data satisfy the compatibility conditions of \eqref{compat_cond}. Let $T >0$. Then there exists a universal constant $\kappa >0$ such that if
\begin{equation}
 \Vert u_0 \Vert^2_{H^2(\Omega)} + \Vert \eta_0 \Vert^2_{H^3(\Sigma)} + \Vert \tilde{c}_0 - c_0 \Vert^2_{H^2(\Sigma)} + \Vert \tilde{b}_0 - b_0 \Vert^2_{H^2(\Omega)} \le \kappa,
\end{equation}
then there exists a unique (strong) solution $(u,p,\eta,c,b)$ to \eqref{perturbation} on the temporal interval $[0,T]$ satisfying the estimate
\begin{equation}\label{lwp_intro_0}
 \sup_{0\le t \le T} \en(t) + \int_0^T \di(t)dt   + \int_0^T \left[\Vert  \dt^2 c(t)\Vert^2_{H^{-1}(\Sigma)} + \Vert \dt^2 b(t) \Vert^2_{(H^1(\Omega)^\ast} \right]dt + \Vert \dt^{2} u \Vert^2_{(\mathcal{X}_T)^*} \ls \en(0).
\end{equation}
Moreover, $\eta$ is such that the mapping $\Phi(\cdot,t)$, defined by \eqref{phi_def}, is a $C^{1}$ diffeomorphism for each $t \in [0,T]$.
\end{theorem}

\begin{remark}
All of the computations involved in the a priori estimates that we develop in this paper are justified by Theorem \ref{lwp_intro}.
\end{remark}

The local existence and a priori estimate then combine to yield the following global-in-time existence and decay result.

\begin{theorem}\label{gwp_intro}
Let $u_0 \in H^2(\Omega)$,   $\eta_0 \in H^3(\Sigma)$, $\tilde{c}_0 \in H^2(\Sigma)$, $\tilde{b}_0 \in H^2(\Omega)$, and assume that $\eta_0$ satisfies \eqref{z_avg} and $(\tilde{c}_0,\tilde{b}_0)$  satisfy \eqref{conserv_c}, where $M > \abs{\Sigma}_2 f(b_0)$ is the equilibrium mass associated to the equilibrium surfactant concentration $(c_0,b_0)$ given in \eqref{equilibrium_config}.   Further assume that the initial data satisfy the compatibility conditions of \eqref{compat_cond}.  Then there exists a universal $\kappa>0$ such that if
\begin{eqnarray}
\tm{u_0}{2}+\tms{\e_0}{3}+\tm{\tilde\b-\b_0}{2}+\tms{\tilde\c-\c_0}{2}\leq\kappa,
\end{eqnarray}
then there exists a unique solution $(u,p,\e,\c,\b)$ to \eqref{perturbation} in the global temporal interval $[0,\infty)$ and a universal constant $\lambda >0$ such that
\begin{eqnarray}
\sup_{t\in[0,\infty)}\ue^{\lambda t}\en(t)+\int_0^{\infty}\di(t)\ud{t}\ls\en(0).
\end{eqnarray}
Moreover, $\e$ is such that the mapping $\Phi$ defined by \eqref{phi_def} is a $C^1$ diffeomorphism.
\end{theorem}
\begin{proof}
This follows from Theorems \ref{aprioris_intro} and \ref{lwp_intro} in a standard way.  See the proof of \cite[Theorem 3]{Kim.Tice2016} for details.
\end{proof}

\begin{remark}
 Theorem \ref{gwp_intro} can be interpreted as an asymptotic stability result: the equilibria $u=0$, $p=0$, $\eta=0$, $\tilde{c}=c_0$ and $\tilde{b} = b_0$ are asymptotically stable, and solutions return to equilibrium exponentially fast.
\end{remark}

\begin{remark}
The surface function $\eta$ is sufficiently small to guarantee that the mapping $\Phi(\cdot,t)$, defined in \eqref{phi_def}, is a diffeomorphism for each $t\ge 0$.  As such, we may change coordinates to $y \in \Omega(t)$ to produce a global-in-time, decaying solution to \eqref{surf}.
\end{remark}

\subsection{Plan of paper and summary of arguments}

Define the horizontal energy
\begin{eqnarray}
\ben&=&\sum_{\abs{\al}\leq 2}\bigg(\half\int_{\Omega}\abs{\p^{\al}u}^2+\half\int_{\Sigma}\abs{\p^{\al}\e}^2
+\frac{\sigma_0}{2}\int_{\Sigma}\abs{\na_{\y}\p^{\al}\e}^2
+\frac{-\sigma_0'f'(\b_0)}{2\c_0}\int_{\Omega}\abs{\p^{\al}\b}^2+\frac{-\sigma_0'}{2\c_0}\int_{\Sigma}\abs{\p^{\al}\c}^2\bigg),
\end{eqnarray}
and the horizontal dissipation
\begin{eqnarray}
\\
\bdi&=&\sum_{\abs{\al}\leq 2}\bigg(\half\int_{\Omega}\abs{\dd\p^{\al}u}^2+\frac{-\beta\sigma'_0f'(\b_0)}{\c_0}\int_{\Omega}\abs{\na \p^{\al}\b}^2+\frac{-\gamma\sigma'_0}{\c_0}\int_{\Sigma}\abs{\na_{\y}\p^{\al}\c}^2+\frac{-\beta\sigma'_0}{\c_0}\int_{\Sigma}\Big(\p^{\al}\c-f'(\b_0)\p^{\al}\b\Big)^2\bigg),\no
\end{eqnarray}
where $\al=(\al_0;\al_1,\al_2,\al_3)$ with $\al_3=0$ is the temporal and spatial derivatives index with $\abs{\al}=2\al_0+\al_1+\al_2+\al_3$.

The global well-posedness and exponential decay follow from a delicate nonlinear energy method. As in \cite{GT2,Kim.Tice2016,WTK,Wu}, we utilize energy and dissipation with higher-order derivatives. This includes the following three steps:
\begin{enumerate}
\item
Horizontal energy-dissipation estimates:  We apply horizontal derivatives (i.e. $\partial_1$, $\partial_2$, and $\partial_t$) to the equations and employ a variant of the basic energy-dissipation equality \eqref{basic_ed} to get estimates.  The horizontal derivatives are needed because they are the only ones compatible with the boundary conditions.  Since the derivatives do not commute with the operators in \eqref{perturbation} we arrive at an equality of the form
\begin{equation}
\dfrac{\ud{\ben}}{\ud{t}}+\bdi = \mathcal{I}
\end{equation}
for a nonlinear interaction term $\mathcal{I}$.
\item
Nonlinear estimates: Here we seek to bound $\mathcal{I}$ by the full energy and dissipation $\en$ and $\di$ in the form  $\abs{\mathcal{I}}\ls \sqrt{\en}\di$, which means $\mathcal{I}$ can be absorbed into $\di$ on the left side of the energy-dissipation equality for small $\en$.  This particular structure is essential for the absorbing argument.
\item
Comparison estimates:  We utilize the horizontal energy and dissipation to control their ``full'' counterparts defined in \eqref{def_E} and \eqref{def_D}.  Up to error terms we show that
\begin{equation}
\en\ls\ben \text{ and }\di\ls\bdi,
\end{equation}
which can then be used to get bounds and decay estimates for the full energy and dissipation.
\end{enumerate}

We follow this general outline here as well, but there are a couple interesting new features in the analysis.

{\bf Averaging estimate of $\c$ and $\b$.}  In order to prove exponential decay we must first prove the coercivity estimate  $\en\ls\di$.  The techniques for estimating the non-surfactant terms here can be found in \cite{Kim.Tice2016}, so it is the pair of surfactant terms that are of primary interest.  Roughly speaking, we need to bound $\tms{\c}{0}$ and $\tm{\b}{0}$ with $\tms{\na_{\y}\c}{0}$ and $\tm{\na\b}{0}$. When the masses of $\tilde\c$ and $\tilde\b$ are conserved separately in the dynamics, this estimate may be proved via the Poincar\'e-Wirtinger inequality and estimates of the averages.  However, the absorption and desorption of surfactant between surface and bulk yield non-constant averages of $\c$ and $\b$.  To get around this problem we introduce an enhanced Poincar\'e-type inequality, which combines the bound of $\c$ and $\b$ in one step, i.e. for $C\in H^1(\Sigma)$ and $B\in H^1(\Omega)$ satisfying
\begin{eqnarray}
\int_{\Sigma}C+\int_{\Omega}B=0,
\end{eqnarray}
we prove
\begin{eqnarray}
\tms{C}{0}+\tm{B}{0}\ls \tms{\na_{\y} C}{0}+\tm{\na B}{0}+\tms{C-f'(\b_0)B}{0},
\end{eqnarray}
where $C$ and $B$ can be variants of $\c$ and $\b$.  We thus see an interesting feature of the dynamics: the conservation of the sum of two quantities combines with the linearization of the flux function $\omega$ to provide a coercivity estimate for both surfactant phases at once.

{\bf Nonlinear estimate of $\omega$.} The linearization \eqref{linearize} of $\omega$ plays a key role in the derivation of the full energy-dissipation structure and the estimate of higher-order derivatives. As Lemma \ref{perturb 01} and \ref{perturb 02} reveal, the quantity $\c-f'(\b_0)\b$ creates a non-negative bound for the interaction between $\c$ and $\b$ and helps also in the above averaging estimate.  Furthermore, the integral form of $\omega(\c_0+\c,\b_0+\b)-\omega(\c_0,\b_0)-\omega_{\c0}\c-\omega_{\b0}\b$ leads to small constants which assist in the absorbing of nonlinear term $\mathcal{I}$.

\subsection{Definitions and terminology}\label{def_and_term}

We now mention some of the definitions, bits of notation, and conventions that we will use throughout the paper.

\textbf{  Einstein summation and constants: }  We will employ the Einstein convention of summing over  repeated indices for vector and tensor operations.  Throughout the paper $C>0$ will denote a generic constant that can depend on the parameters of the problem, and $\Omega$, but does not depend on the data, etc.  We refer to such constants as ``universal.''  They are allowed to change from one inequality to the next.  We will employ the notation $a \ls b$ to mean that $a \le C b$ for a universal constant $C>0$.

\textbf{  Norms: } We write $H^k(\Omega)$ with $k\ge 0$ and  $H^s(\Sigma)$ with $s \in \mathbb{R}$ for the usual Sobolev spaces.  We will typically write $H^0 = L^2$.  To avoid notational clutter, we will avoid writing $H^k(\Omega)$ or $H^k(\Sigma)$ in our norms and typically write only $\Vert \cdot \Vert_{k}$ for $H^k(\Omega)$ norms and $\Vert\cdot\Vert_{\Sigma,s}$ for $H^s(\Sigma)$ norms.

\section{Perturbation form}

In this section, we introduce two types of formulation for system \eqref{perturbation} which will assist the estimate of temporal and horizontal derivatives.

\subsection{Geometric perturbed form}

We consider the system for $(v,q,\xi,h,d)$ as
\begin{equation}\label{perturb equation 1}
\left\{
\begin{array}{ll}
\dt v-\dt\be\tilde LK\p_3 v+u\cdot\nl v+\nl\cdot\sl(v,q)=F^1&\ \ \text{in}\ \ \Omega,\\
\nl\cdot v=F^2 &\ \ \text{in}\ \ \Omega,\\
\s(v,q)\n=\xi\n-\sigma_0\de_{\y}\xi\n-\sigma'_0\na_{\y}h+F^3&\ \ \text{on}\ \ \Sigma,\\
v=0&\ \ \text{on}\ \ \Sigma_b,\\
\dt\xi=v\cdot\n+F^4 &\ \ \text{on}\ \ \Sigma,\\
\ \\
\dt h+\c_0\na_{\y}\cdot v=\gamma\de_{\y}h-\beta\nl d\cdot\n+F^5 &\ \ \text{on}\ \ \Sigma,\\
\dt d-\dt\be\tilde LK\p_3 d+u\cdot\nl d=\beta\ll d+F^6&\ \ \text{in}\ \ \Omega,\\
\nl d\cdot\n=\omega_{\c}h+\omega_{\b0}d+F^7&\ \ \text{on}\ \ \Sigma,\\
\p_3d=0&\ \ \text{on}\ \ \Sigma_b.
\end{array}
\right.
\end{equation}

\begin{lemma}\label{perturb 01}
We have
\begin{eqnarray}
&&\frac{\ud{}}{\ud{t}}\bigg(\half\int_{\Omega}J\abs{v}^2+\half\int_{\Sigma}\abs{\xi}^2
+\frac{\sigma_0}{2}\int_{\Sigma}\abs{\na_{\y}\xi}^2+\frac{-\sigma_0'f'(\b_0)}{2\c_0}\int_{\Omega}J\abs{d}^2+\frac{-\sigma_0'}{2\c_0}\int_{\Sigma}\abs{h}^2\bigg)\\
&&+\bigg(\half\int_{\Omega}J\abs{\dl v}^2+\frac{-\beta\sigma'_0f'(\b_0)}{\c_0}\int_{\Omega}J\abs{\nl d}^2+\frac{-\gamma\sigma'_0}{\c_0}\int_{\Sigma}\abs{\na_{\y}h}^2\bigg)\no\\
&&+\frac{-\beta\sigma'_0}{\c_0}\int_{\Sigma}\Big(h-f'(\b_0)d\Big)(\omega_{\c0}h+\omega_{\b0}d)\no\\
&=&\int_{\Omega}Jv\cdot F^1+\int_{\Omega}JqF^2-\int_{\Sigma}v\cdot F^3+\int_{\Sigma}(\xi-\sigma_0\de_{\y}\xi)F^4\no\\
&&+\frac{-\sigma'_0}{\c_0}\int_{\Sigma}hF^5+\frac{-\sigma'_0f'(\b_0)}{\c_0}\int_{\Omega}JdF^6
+\frac{\beta\sigma'_0}{\c_0}\int_{\Sigma}\Big(h-f'(\b_0)d\Big)F^7.\no
\end{eqnarray}
\end{lemma}
\begin{proof}
Multiplying $Jv$ on both sides of the Navier-Stokes equation and integrating over $\Omega$, we have
\begin{eqnarray}
\int_{\Omega}Jv\cdot\dt v+\int_{\Omega}Jv\cdot\Big(-\dt\be\tilde LK\p_3 v+u\cdot\nl v\Big)+\int_{\Omega}Jv\cdot\Big(\nl\cdot\sl(v,q)\Big)=\int_{\Omega}Jv\cdot F^1.
\end{eqnarray}
We may directly compute
\begin{eqnarray}
&&\frac{\ud{}}{\ud{t}}\bigg(\int_{\Omega}Jv\cdot\dt v\bigg)+\int_{\Omega}Jv\cdot\Big(-\dt\be\tilde LK\p_3 v+u\cdot\nl v\Big)\\
&=&\frac{\ud{}}{\ud{t}}\bigg(\half\int_{\Omega}J\abs{v}^2\bigg)-\half\int_{\Omega}\dt J\abs{v}^2+\half\int_{\Omega}\p_3(\dt\be\tilde L)\abs{v}^2-\int_{\Sigma}\dt\e\abs{v^2}\no\\
&&-\int_{\Omega}(\nl\cdot u)\abs{v^2}+\int_{\Sigma}(u\cdot\n)\abs{v^2}\no\\
&=&\frac{\ud{}}{\ud{t}}\bigg(\half\int_{\Omega}J\abs{v}^2\bigg).\no
\end{eqnarray}
Also, we know
\begin{eqnarray}
&&\int_{\Omega}Jv\cdot\Big(\nl\cdot\sl(v,q)\Big)\\
&=&-\int_{\Omega}Jq(\nl\cdot v)+\half\int_{\Omega}J\abs{\dl v}^2+\int_{\Sigma}\Big(\s(v,q)\n\Big)\cdot v\no\\
&=&-\int_{\Omega}JqF^2+\half\int_{\Omega}J\abs{\dl v}^2+\int_{\Sigma}\Big(\xi\n-\sigma_0\de_{\y}\xi\n\Big)\cdot v-\sigma'_0\int_{\Sigma}\na_{\y}h\cdot v+\int_{\Sigma}F^3\cdot v.\no
\end{eqnarray}
We may continue computing
\begin{eqnarray}
\int_{\Sigma}\Big(\xi\n-\sigma_0\de_{\y}\n\Big)\cdot v&=&\int_{\Sigma}(\xi-\sigma_0\de_{\y}\xi)(\dt\xi-F^4)\\
&=&\frac{\ud{}}{\ud{t}}\bigg(\half\int_{\Sigma}\abs{\xi}^2
+\frac{\sigma_0}{2}\int_{\Sigma}\abs{\na_{\y}\xi}^2\bigg)+\int_{\Sigma}(\xi-\sigma_0\de_{\y}\xi)F^4.\no
\end{eqnarray}
In surface surfactant equation, multiplying $\dfrac{-\sigma_0'}{\c_0}h$ on both sides and integrating over $\Sigma$, we have
\begin{eqnarray}
\frac{\ud{}}{\ud{t}}\bigg(\frac{-\sigma_0'}{2\c_0}\int_{\Sigma}\abs{h}^2\bigg)+\sigma'_0\int_{\Sigma}\na_{\y}h\cdot v+\frac{-\gamma\sigma'_0}{\c_0}\int_{\Sigma}\abs{\na_{\y}h}^2+\frac{-\beta\sigma'_0}{\c_0}\int_{\Sigma}h(\nl d\cdot\n)=\frac{-\sigma'_0}{\c_0}\int_{\Sigma}hF^5.
\end{eqnarray}
In surface surfactant equation, multiplying $\dfrac{-\beta\sigma'_0f'(\b_0)}{\c_0}Jd$ on both sides and integrating over $\Omega$, we have
\begin{eqnarray}
\frac{\ud{}}{\ud{t}}\bigg(\frac{-\sigma_0'f'(\b_0)}{2\c_0}\int_{\Omega}J\abs{d}^2\bigg)+\frac{-\beta\sigma'_0f'(\b_0)}{\c_0}\int_{\Omega}J\abs{\nl d}^2-
\frac{-\beta\sigma'_0}{\c_0}\int_{\Sigma}f'(\b_0)d(\nl d\cdot\n)=\int_{\Omega}JdF^6.
\end{eqnarray}
Note that
\begin{eqnarray}
\int_{\Sigma}\Big(h-f'(\b_0)d\Big)(\nl d\cdot\n)=\int_{\Sigma}\Big(h-f'(\b_0)d\Big)(\omega_{\c}h+\omega_{\b0}d)
+\int_{\Sigma}\Big(h-f'(\b_0)d\Big)F^7.
\end{eqnarray}
Summarizing all above, we have the desired result.
\end{proof}

We will mainly use this form to study the estimate of temporal derivative. We apply $\dt$ operator on the unknown variables and let $(v,q,\xi,h,d)=(\dt u,\dt p,\dt\e,\dt\c,\dt\b)$. In this situation, we have
\begin{eqnarray}
F^1&=&F^{1,1}+F^{1,2}+F^{1,3}+F^{1,4},\\
F^{1,1}_i&=&\dt(\dt\be\tilde LK)\p_3u_i,\\
F^{1,2}_i&=&-\dt(u_j\a_{jk})\p_ku_i+\dt\a_{ik}\p_kp,\\
F^{1,3}_i&=&\dt\a_{jl}\p_l(\a_{im}\p_mu_j+\a_{jm}\p_mu_i),\\
F^{1,4}_i&=&\a_{jk}\p_k(\dt\a_{il}\p_lu_j+\dt\a_{jl}\p_lu_i),
\end{eqnarray}
\begin{eqnarray}
F^2&=&-\dt\a_{ij}\p_ju_i,
\end{eqnarray}
\begin{eqnarray}
F^3&=&F^{3,1}+F^{3,2}+F^{3,3},\\
F^{3,1}_i&=&(\e-p)\dt\n_i+(\a_{ik}\p_ku_j+\a_{jk}\p_ku_i)\dt\n_j+(\dt\a_{ik}\p_ku_j+\dt\a_{jk}\p_ku_i)\n_j,\\
F^{3,2}_i&=&-\sigma'(\tilde\c)\dt\c H\n_i-(\sigma(\tilde\c)-\sigma_0)\dt H\n_i-(\sigma_0\dt H-\sigma_0\dt\de_{\y}\e)\n_i-\sigma(\tilde\c)H\dt\n_i,\\
F^{3,3}_i&=&-\frac{\na_{\y}\e\cdot\na_{\y}\dt\e}{\sqrt{1+\abs{\na_{\y}\e}^2}}\sigma'(\tilde\c)-\sqrt{1+\abs{\na_{\y}\e}^2}\sigma''(\tilde\c)\dt\tilde\c(\na_{\Gamma}\tilde\c)_i
+\sqrt{1+\abs{\na_{\y}\e}^2}\sigma'(\tilde\c)\nu_i(\nu_{\y}\cdot\na_{\y})\dt\c\\
&&-\sqrt{1+\abs{\na_{\y}\e}^2}\sigma'(\tilde\c)\Big(\dt\nu_i(\nu_{\y}\cdot\na_{\y})\tilde\c+\nu_i(\dt\nu_{\y}\cdot\na_{\y})\tilde\c\Big),\no
\end{eqnarray}
\begin{eqnarray}
F^4&=&\dt\na_{\y}\e\cdot u,
\end{eqnarray}
\begin{eqnarray}
F^5&=&F^{5,1}+F^{5,2},\\
F^{5,1}&=&\dt\Big(-u\cdot\na_{\y}\c-\c\na_{\Gamma}\cdot u+\gamma(\de_{\Gamma}\c-\de_{\y}\c)-\c_0(\na_{\Gamma}\cdot u-\na_{\y}\cdot u)\Big),\\
F^{5,2}&=&-\beta(\dt\a_{ik})\p_k\b\n_i-\beta\a_{ik}\p_k\b(\dt\n_i),
\end{eqnarray}
\begin{eqnarray}
F^6&=&F^{6,1}+F^{6,2}+F^{6,3},\\
F^{6,1}_i&=&\dt(\dt\be\tilde LK)\p_3\b,\\
F^{6,2}_i&=&-\dt(u_j\a_{jk})\p_k\b,\\
F^{6,3}_i&=&\beta\dt\a_{kl}\p_k\a_{jm}\p_m\b+\beta\a_{jk}\p_k\dt\a_{ml}\p_m\b,
\end{eqnarray}
\begin{eqnarray}
F^7&=&F^{7,1}+F^{7,2},\\
F^{7,1}&=&-(\dt\a_{ik})\p_k\b\n_i-\a_{ik}\p_k\b(\dt\n_i),\\
F^{7,2}&=&\dt\Big(\omega(\tilde\c,\tilde\b)-\omega_{\c0}\c-\omega_{\b0}\b\Big).
\end{eqnarray}

\subsection{Linear perturbed form}

We consider the system for $(v,q,\xi,h,d)$ as
\begin{equation}\label{perturb equation 2}
\left\{
\begin{array}{ll}
\dt v-\de v+\na q=Z^1&\ \ \text{in}\ \ \Omega,\\
\na\cdot v=Z^2 &\ \ \text{in}\ \ \Omega,\\
(qI-\dd v)e_3=\xi e_3-\sigma_0\de_{\y}\xi e_3-\sigma'_0\na_{\y}h+Z^3&\ \ \text{on}\ \ \Sigma,\\
v=0&\ \ \text{on}\ \ \Sigma_b,\\
\dt\xi=v_3+Z^4 &\ \ \text{on}\ \ \Sigma,\\
\ \\
\dt h+\c_0\na_{\y}\cdot v=\gamma\de_{\y}h-\beta\p_3d+Z^5 &\ \ \text{on}\ \ \Sigma,\\
\dt d=\beta\de d+Z^6&\ \ \text{in}\ \ \Omega,\\
\p_3d=\omega_{\c0}h+\omega_{\b0}d+Z^7&\ \ \text{on}\ \ \Sigma,\\
\p_3d=0&\ \ \text{on}\ \ \Sigma_b.
\end{array}
\right.
\end{equation}

\begin{lemma}\label{perturb 02}
We have
\begin{eqnarray}
&&\frac{\ud{}}{\ud{t}}\bigg(\half\int_{\Omega}\abs{v}^2+\half\int_{\Sigma}\abs{\xi}^2
+\frac{\sigma_0}{2}\int_{\Sigma}\abs{\na_{\y}\xi}^2+\frac{-\sigma_0'f'(\b_0)}{2\c_0}\int_{\Omega}\abs{d}^2+\frac{-\sigma_0'}{2\c_0}\int_{\Sigma}\abs{h}^2\bigg)\\
&&+\bigg(\half\int_{\Omega}\abs{\dd v}^2+\frac{-\beta\sigma'_0f'(\b_0)}{\c_0}\int_{\Omega}\abs{\na d}^2+\frac{-\gamma\sigma'_0}{\c_0}\int_{\Sigma}\abs{\na_{\y}h}^2\bigg)
+\frac{-\beta\sigma'_0}{\c_0}\int_{\Sigma}\Big(h-f'(\b_0)d\Big)(\omega_{\c0}h+\omega_{\b0}d)\no\\
&=&\int_{\Omega}v\cdot Z^1+\int_{\Omega}(qZ^2-v\cdot\na Z^2)-\int_{\Sigma}v\cdot Z^3+\int_{\Sigma}(\xi-\sigma_0\de_{\y}\xi)Z^4\no\\
&&+\frac{-\sigma'_0}{\c_0}\int_{\Sigma}hZ^5+\frac{-\sigma'_0f'(\b_0)}{\c_0}\int_{\Omega}dZ^6
+\frac{\beta\sigma'_0}{\c_0}\int_{\Sigma}\Big(h-f'(\b_0)d\Big)Z^7.\no
\end{eqnarray}
\end{lemma}
\begin{proof}
From the Navier-Stokes equation, we have
\begin{eqnarray}
\dt v+\na\cdot\s(v,q)=\dt v-\de v+\na q-\na Z^2=Z^1-\na Z^2.
\end{eqnarray}
Multiplying $v$ on both sides of Navier-Stokes equation and integrating over $\Omega$, we have
\begin{eqnarray}
&&\frac{\ud{}}{\ud{t}}\bigg(\half\int_{\Omega}\abs{v}^2\bigg)+\half\int_{\Omega}\abs{\dd v}^2+\int_{\Sigma}\Big(\xi -\sigma_0\de_{\y}\xi\Big)v_3-\sigma'_0\int_{\Sigma}\na_{\y}h\cdot v\\
&=&\int_{\Omega}v\cdot Z^1+\int_{\Omega}(qZ^2-v\cdot\na Z^2)-\int_{\Sigma}v\cdot Z^3.\no
\end{eqnarray}
Also, we have
\begin{eqnarray}
\int_{\Sigma}\Big(\xi -\sigma_0\de_{\y}\xi\Big)v_3&=&\int_{\Sigma}\Big(\xi -\sigma_0\de_{\y}\xi\Big)(\dt\xi-Z^4)\\
&=&\frac{\ud{}}{\ud{t}}\bigg(\half\int_{\Sigma}\abs{\xi}^2
+\frac{\sigma_0}{2}\int_{\Sigma}\abs{\na_{\y}\xi}^2\bigg)-\int_{\Sigma}(\xi-\sigma_0\de_{\y}\xi)Z^4.\no
\end{eqnarray}
In surface surfactant equation, multiplying $\dfrac{-\sigma_0'}{\c_0}h$ on both sides and integrating over $\Sigma$, we have
\begin{eqnarray}
\frac{\ud{}}{\ud{t}}\bigg(\frac{-\sigma_0'}{2\c_0}\int_{\Sigma}\abs{h}^2\bigg)+\sigma'_0\int_{\Sigma}\na_{\y}h\cdot v+\frac{-\gamma\sigma'_0}{\c_0}\int_{\Sigma}\abs{\na_{\y}h}^2+\frac{-\beta\sigma'_0}{\c_0}\int_{\Sigma}h\p_3d=\frac{-\sigma'_0}{\c_0}\int_{\Sigma}hZ^5.
\end{eqnarray}
In surface surfactant equation, multiplying $\dfrac{-\beta\sigma'_0f'(\b_0)}{\c_0}d$ on both sides and integrating over $\Omega$, we have
\begin{eqnarray}
\frac{\ud{}}{\ud{t}}\bigg(\frac{-\sigma_0'f'(\b_0)}{2\c_0}\int_{\Omega}\abs{d}^2\bigg)+\frac{-\beta\sigma'_0f'(\b_0)}{\c_0}\int_{\Omega}\abs{\na d}^2-
\frac{-\beta\sigma'_0}{\c_0}\int_{\Sigma}f'(\b_0)d\p_3d=\int_{\Omega}dZ^6.
\end{eqnarray}
Note that
\begin{eqnarray}
\int_{\Sigma}\Big(h-f'(\b_0)d\Big)\p_3d=\int_{\Sigma}\Big(h-f'(\b_0)d\Big)(\omega_{\c0}h+\omega_{\b0}d)
+\int_{\Sigma}\Big(h-f'(\b_0)d\Big)Z^7.
\end{eqnarray}
Summarizing all above, we have the desired result.
\end{proof}

We will mainly use this form to study the estimate of spatial derivative. In the system \eqref{perturbation}, we have
\begin{equation}
\left\{
\begin{array}{ll}
\dt u-\de u+\na p=G^1&\ \ \text{in}\ \ \Omega,\\
\na\cdot u=G^2 &\ \ \text{in}\ \ \Omega,\\
(pI-\dd u)e_3=\e e_3-\sigma_0\de_{\y}\e e_3-\sigma'_0\na_{\y}\c+G^3&\ \ \text{on}\ \ \Sigma,\\
u=0&\ \ \text{on}\ \ \Sigma_b,\\
\dt\xi=u_3+G^4 &\ \ \text{on}\ \ \Sigma,\\
\ \\
\dt\c+\c_0\na_{\y}\cdot u=\de_{\y}\c-\p_3\b+G^5 &\ \ \text{on}\ \ \Sigma,\\
\dt\b=\de\b+G^6&\ \ \text{in}\ \ \Omega,\\
\p_3\b=\omega_{\c0}\c+\omega_{\b0}\b+G^7&\ \ \text{on}\ \ \Sigma,\\
\p_3\b=0&\ \ \text{on}\ \ \Sigma_b,
\end{array}
\right.
\end{equation}
where
\begin{eqnarray}
G^1&=&G^{1,1}+G^{1,2}+G^{1,3}+G^{1,4}+G^{1,5},\\
G^{1,1}_i&=&(\delta_{ij}-\a_{ij})\p_jp,\\
G^{1,2}_i&=&u_j\a_{jk}\p_ku_i,\\
G^{1,3}_i&=&\Big(K^2(1+A^2+B^2)-1\Big)\p_{33}u_i-2AK\p_{13}u_i-2BK\p_{23}u_i,\\
G^{1,4}_i&=&\Big(-K^3(1+A^2+B^2)\p_3J+AK^2(\p_1J+\p_3A)+BK^2(\p_2J+\p_3B)-K(\p_1A+\p_2B)\Big)\p_3u_i,\\
G^{1,5}_i&=&\dt\be\tilde LK\p_3u_i,
\end{eqnarray}
\begin{eqnarray}
G^2&=&AK\p_3u_1+BK\p_3u_2+(1-K)\p_3u_3,
\end{eqnarray}
\begin{eqnarray}
G^3&=&G^{3,1}+G^{3,2}+G^{3,3}+G^{3,4},\\
G^{3,1}&=&\p_1\e\left(\begin{array}{l}
p-\e-2(\p_1u_1-AK\p_3u_1)\\
-\p_2u_1-\p_1u_2+BK\p_3u_1+AK\p_3u_2\\
-\p_1u_3-K\p_3u_1+AK\p_3u_3
\end{array}\right)\\
&&+\p_2\e\left(\begin{array}{l}
-\p_2u_1-\p_1u_2+BK\p_3u_1+AK\p_3u_2\\
p-\e-2(\p_2u_2-BK\p_3u_2)\\
-\p_2u_3-K\p_3u_2+BK\p_3u_3
\end{array}\right)
+\left(\begin{array}{l}
(K-1)\p_3u_1+AK\p_3u_3\\
(K-1)\p_3u_2+BK\p_3u_3\\
2(K-1)\p_3u_3
\end{array}\right),\no\\
G^{3,2}&=&(\sigma(\tilde\c)-\sigma_0)\de_{\y}\e e_3+\sigma(\tilde\c)(H-\de_{\y}\e)\n+\sigma(\tilde\c)\de_{\y}\e(\n-e_3),\\
G^{3,3}&=&\Big(\sqrt{1+\abs{\na_{\y}\e}^2}-1\Big)\sigma'(\tilde\c)\na_{\y}\c+(\sigma'(\tilde\c)-\sigma'_0)\na_{\y}\c
+\sqrt{1+\abs{\na_{\y}\e}^2}\sigma'(\tilde\c)(\na_{\Gamma}\c-\na_{\y}\c),\\
G^{3,4}&=&\sigma'(\tilde\c)\nu_{\y}\cdot\na_{\y}\c e_3,
\end{eqnarray}
\begin{eqnarray}
G^4&=&-\p_1\e u_1-\p_2\e u_2,
\end{eqnarray}
\begin{eqnarray}
G^5&=&G^{5,1}+G^{5,2},\\
G^{5,1}&=&-u\cdot\na_{\y}\c-\c\na_{\Gamma}\cdot u+\gamma(\de_{\Gamma}\c-\de_{\y}\c)-\c_0(\na_{\Gamma}\cdot u-\na_{\y}\cdot u),\\
G^{5,2}&=&\beta(\p_1\b-AK\p_3\b)\p_1\e+\beta(\p_2\b-BK\p_3\b)\p_2\e+\beta(K-1)\p_3\b,
\end{eqnarray}
\begin{eqnarray}
G^6&=&G^{6,1}+G^{6,2}+G^{6,3}+G^{6,4},\\
G^{6,1}_i&=&u_j\a_{jk}\p_k\b,\\
G^{6,2}_i&=&\beta\Big(K^2(1+A^2+B^2)-1\Big)\p_{33}\b-2AK\p_{13}\b-2BK\p_{23}\b,\\
G^{6,3}_i&=&\beta\Big(-K^3(1+A^2+B^2)\p_3J+AK^2(\p_1J+\p_3A)+BK^2(\p_2J+\p_3B)-K(\p_1A+\p_2B)\Big)\p_3\b,\\
G^{6,4}_i&=&\dt\be\tilde LK\p_3\b,
\end{eqnarray}
\begin{eqnarray}
G^7&=&G^{7,1}+G^{7,2},\\
G^{7,1}&=&(\p_1\b-AK\p_3\b)\p_1\e+(\p_2\b-BK\p_3\b)\p_2\e+(K-1)\p_3\b,\\
G^{7,1}&=&\omega(\tilde\c,\tilde\b)-\omega_{\c0}\c-\omega_{\b0}\b.
\end{eqnarray}

\section{Preliminary estimates}

In this section, we mainly cover some preliminary estimates in energy-dissipation estimates, nonlinear estimates and comparison estimates.

\subsection{A basic lemma}

The following lemma provides basic estimates for terms that appear frequently in our analysis.

\begin{lemma}\label{nonlinear lemma 1}
There exists a universal $0<\delta\le 1$ such that if $\tms{\e}{\frac{5}{2}}\leq\delta$, then the following holds
\begin{enumerate}
\item
We have
\begin{eqnarray}
\nm{J-1}_{L^{\infty}(\Omega)}+\nm{A}_{L^{\infty}(\Omega)}+\nm{B}_{L^{\infty}(\Omega)}&\leq&\frac{1}{4},\\
\nm{K}_{L^{\infty}(\Omega)}+\nm{\a}_{L^{\infty}(\Omega)}&\leq&1.
\end{eqnarray}

\item
The mapping $\Phi$ is a diffeomorphism from $\Omega$ to $\Omega(t)$.
\item
For all $v\in H^1(\Omega)$ with $v=0$ on $\Sigma_b$, we have
\begin{eqnarray}
\int_{\Omega}\abs{\dd v}^2&\leq&\int_{\Omega}J\abs{\dl v}^2+C_0\sqrt{\en}\int_{\Omega}\abs{\dd v}^2.
\end{eqnarray}

\item
For all $\psi \in H^1(\Omega)$ we have
\begin{eqnarray}
\int_{\Omega} \abs{\na \psi}^2 & \leq& \int_{\Omega}J \abs{\nl \psi}^2 + C_0\sqrt{\en}\int_{\Omega}\abs{\na \psi}^2 .
\end{eqnarray}

\item We have the pointwise bounds
\begin{eqnarray}
-\frac{\c_0}{2}\leq\c\leq\frac{\c_0}{2},&\ \ &\frac{\c_0}{2}\leq\tilde\c\leq\frac{3\c_0}{2},\\
-\frac{\b_0}{2}\leq\b\leq\frac{\b_0}{2},&\ \ &\frac{\b_0}{2}\leq\tilde\b\leq\frac{3\b_0}{2}.
\end{eqnarray}
\end{enumerate}
\end{lemma}
\begin{proof}
See the proof of \cite[Lemma 4.1]{Kim.Tice2016} with obvious modifications. The last inequality involving $\b_0$ can be proved as for $\c_0$.
\end{proof}

\subsection{Temporal estimates}

Now we turn our attention to estimating various nonlinearities that appear when we consider the temporally-differentiated problem.

\begin{lemma}\label{nonlinear lemma 2}
Suppose $\en\leq\delta<<1$. We have
\begin{eqnarray}\label{ct 04}
\tm{JF^1}{0}+\tms{F^3}{0}+\tms{F^4}{0}+\tm{JF^6}{0}&\leq&\sqrt{\en}\sqrt{\di},
\end{eqnarray}
\begin{eqnarray}\label{ct 05}
\abs{\int_{\Omega}JF^2\dt p-\dt\int_{\Omega}JF^2p}&\ls&\sqrt{\en}\di,\ \ \abs{\int_{\Omega}JF^2p}\ls \en^{\frac{3}{2}},
\end{eqnarray}
\begin{eqnarray}\label{ct 06}
\abs{\int_{\Sigma}F^5\dt\c}\ls \sqrt{\en}\di,
\end{eqnarray}
\begin{eqnarray}\label{ct 07}
\tms{F^7}{0}\ls\sqrt{\en}\sqrt{\di}.
\end{eqnarray}
\end{lemma}
\begin{proof}
\eqref{ct 04} and \eqref{ct 05} can be proved as \cite[Theorem 4.2]{Kim.Tice2016}. For \eqref{ct 06}, we already know the estimate of $F^{5,1}$, so we only need to bound
\begin{eqnarray}
F^{5,2}=-\beta(\dt\a_{ik})\p_k\b\n_i-\beta\a_{ik}\p_k\b(\dt\n_i).
\end{eqnarray}
We directly estimate
\begin{eqnarray}
\\
\tms{F^{5,2}}{0}&\ls&\nm{\dt\na\be}_{H^0(\Sigma)}\nm{\na\b}_{L^{\infty}(\Sigma)}\Big(1+\nm{\na_{\y}\e}_{L^{\infty}(\Sigma)}\Big)
+\nm{\dt\na_{\y}\e}_{H^0(\Sigma)}\nm{\na\b}_{L^{\infty}(\Sigma)} \Big( 1+\nm{\na\be}_{L^{\infty}(\Sigma)}\Big) \no\\
&\ls&\nm{\dt\be}_{H^{\frac{3}{2}}(\Omega)}\tms{\na\b}{\frac{3}{2}}\Big(1+\tms{\na_{\y}\e}{\frac{3}{2}}\Big)+\tms{\dt\e}{1}\tms{\na\b}{\frac{3}{2}} \Big( 1+\tms{\na\be}{\frac{3}{2}} \Big)\no\\
&\ls&\tms{\dt\e}{1}\tm{\b}{3}\Big(1+\tms{\e}{\frac{5}{2}}\Big)
\ls \sqrt{\en}\sqrt{\di}.
\end{eqnarray}
Thus,
\begin{eqnarray}
\abs{\int_{\Sigma}F^{5,2}\dt\c}&\ls&\tms{F^{5,2}}{0}\tms{\dt\c}{0}\ls \sqrt{\en}\di.
\end{eqnarray}
For \eqref{ct 07}, we need to estimate
\begin{eqnarray}
F^7&=&F^{7,1}+F^{7,2},\\
F^{7,1}&=&-(\dt\a_{ik})\p_k\b\n_i-\a_{ik}\p_k\b(\dt\n_i),\\
F^{7,2}&=&\dt\Big(\omega(\tilde\c,\tilde\b)-\omega_{\c0}\c-\omega_{\b0}\b\Big).
\end{eqnarray}
It is easy to see
\begin{eqnarray}
\tms{F^{7,1}}{0}=\beta\tms{F^{5,2}}{0}\ls \sqrt{\en}\sqrt{\di}.
\end{eqnarray}
On the other hand, using Taylor expansion as in \eqref{linearize}, we obtain
\begin{eqnarray}
\tms{F^{7,2}}{0}&\ls&  \Big(\tms{\dt\c}{0}+\tms{\dt\b}{0}\Big)\Big(\nm{\c}_{L^{\infty}(\Sigma)}+\nm{\b}_{L^{\infty}(\Sigma)} + \nm{\c}_{L^{\infty}(\Sigma)}^2 +\nm{\b}_{L^{\infty}(\Sigma)}^2\Big)  \\
&\ls&\Big(\tms{\dt\c}{0}+\tm{\dt\b}{1}\Big)\Big(\tms{\c}{\frac{3}{2}}+\tm{\b}{2} + \tms{\c}{\frac{3}{2}}^2 +\tm{\b}{2}^2\Big)\no\\
&\ls& \sqrt{\en}\sqrt{\di}.\no
\end{eqnarray}
Therefore, we know
\begin{eqnarray}
\tms{F^7}{0}\ls \sqrt{\en}\sqrt{\di}.
\end{eqnarray}
\end{proof}

\subsection{Spatial estimates}

Now we control various nonlinearities that appear in the spatially-differentiated problem.

\begin{lemma}\label{nonlinear lemma 3}
Suppose $\en\leq\delta<<1$. We have
\begin{eqnarray}\label{ct 08}
\\
\tm{G^1}{1}+\tm{G^2}{2}+\tms{G^3}{\frac{3}{2}}+\tms{G^4}{\frac{5}{2}}
+\tms{G^5}{1}+\tm{G^6}{1}+\tms{G^7}{\frac{3}{2}}&\ls& \sqrt{\en}\sqrt{\di},\no
\end{eqnarray}
\begin{eqnarray}\label{ct 09}
\\
\tm{G^1}{0}+\tm{G^2}{1}+\tms{G^3}{\frac{1}{2}}+\tms{G^4}{\frac{3}{2}}
+\tms{G^5}{0}+\tm{G^6}{0}+\tms{G^7}{\frac{1}{2}}&\ls& \en.\no
\end{eqnarray}
\begin{eqnarray}\label{ct 10}
\tms{\dt G^4}{\frac{1}{2}}&\ls& \sqrt{\en}\sqrt{\di},
\end{eqnarray}
\end{lemma}
\begin{proof}
In \eqref{ct 08}, the estimates of $G^1$, $G^2$, $G^3$, $G^4$, $\dt G^4$, $G^{5,1}$ and $G^6$ can be proved as \cite[Theorem 4.3]{Kim.Tice2016}. It suffices to estimate $G^7=G^{7,1}+G^{7,2}$ since $G^{5,2}=\beta G^{7,1}$. By \eqref{ap 01} with $r=s_1=s_2=\dfrac{3}{2}$, we have
\begin{eqnarray}
\tms{G^{7,1}}{\frac{3}{2}}&=&\tms{(\p_1\b-AK\p_3\b)\p_1\e+(\p_2\b-BK\p_3\b)\p_2\e+(K-1)\p_3\b}{\frac{3}{2}}\\
&\ls&\Big(1+\tms{\na\be}{\frac{3}{2}}\Big)\tms{\na\b}{\frac{3}{2}}\tms{\na\e}{\frac{3}{2}}+\tms{\na\be}{\frac{3}{2}}\tms{\na\b}{\frac{3}{2}}\no\\
&\ls&\tms{\e}{\frac{5}{2}}\tm{\b}{3}\ls \sqrt{\en}\sqrt{\di}.\no
\end{eqnarray}
Using Taylor expansion as \eqref{linearize}, we obtain
\begin{eqnarray}
\tms{G^{7,2}}{\frac{3}{2}}&=&\tms{\omega(\tilde\c,\tilde\b)-\omega_{\c0}\c-\omega_{\b0}\b}{\frac{3}{2}}\\
&\ls&\nm{\omega}_{C^4}\Big(\tms{\c}{3}+\tm{\b}{3}\Big)\Big(\tms{\c}{2}+\tm{\b}{2}\Big)^2\abs{\int_0^1s(1-s)\ud{s}}\no\\
&\ls&\sqrt{\en}\sqrt{\di}.\no
\end{eqnarray}
Similarly, in \eqref{ct 09}, the estimates of $G^1$, $G^2$, $G^3$, $G^4$, $G^{5,1}$ and $G^6$ can be proved as before. By \eqref{ap 02} with $r=s_1=\dfrac{1}{2}$ and $s_2=2$, we estimate $G^7$ as
\begin{eqnarray}
\tms{G^{7,1}}{\frac{1}{2}}&=&\tms{(\p_1\b-AK\p_3\b)\p_1\e+(\p_2\b-BK\p_3\b)\p_2\e+(K-1)\p_3\b}{\frac{1}{2}}\\
&\ls&\Big(1+\tms{\na\be}{2}\Big)\tms{\na\b}{\frac{1}{2}}\tms{\na\e}{2}+\tms{\na\be}{2}\tms{\na\b}{\frac{1}{2}}\no\\
&\ls&\tms{\e}{3}\tm{\b}{2}\ls \en.\no
\end{eqnarray}
Also, using Taylor expansion as \eqref{linearize}, we get
\begin{eqnarray}
\tms{G^{7,2}}{\frac{1}{2}}&=&\tms{\omega(\tilde\c,\tilde\b)-\omega_{\c0}\c-\omega_{\b0}\b}{\frac{1}{2}}\\
&\ls&\nm{\omega}_{C^3}\Big(\tms{\c}{2}+\tm{\b}{2}\Big)^3\abs{\int_0^1s(1-s)\ud{s}}\ls\en.\no
\end{eqnarray}
\end{proof}

\section{Energy-dissipation estimates}

In this section, we begin to investigate the energy-dissipation structure for temporal and horizontal derivatives.  Define
\begin{eqnarray}
\f=\int_{\Omega}JF^2p  +  \half\int_{\Omega}(J-1)\abs{\dt u}^2 +\frac{-\sigma_0'f'(\b_0)}{2\c_0}\int_{\Omega}(J-1)\abs{\dt\b}^2.
\end{eqnarray}
\begin{theorem}\label{energy theorem}
Suppose
\begin{eqnarray}
\sup_{t\in[0,T]}\en(t)\leq\delta << 1,\ \ \int_0^T\di(t)\ud{t}<\infty.
\end{eqnarray}
We have
\begin{eqnarray}
\frac{\ud{}}{\ud{t}}(\ben-\f)+\bdi\ls\sqrt{\en}\di.
\end{eqnarray}
\end{theorem}
\begin{proof}
We apply operator $\p^{\al}$ for $\al=(\al_0,\al_1,\al_2,0)$ to the system \eqref{perturbation} and estimate each term. We divide it into several steps:\\
\ \\
Step 1: Temporal derivatives.\\
We apply $\dt$ to system \eqref{perturbation}. Then we have
\begin{eqnarray}\label{energy theorerm temporal}
&&\frac{\ud{}}{\ud{t}}\bigg(\half\int_{\Omega}J\abs{\dt u}^2+\half\int_{\Sigma}\abs{\dt\e}^2
+\frac{\sigma_0}{2}\int_{\Sigma}\abs{\na_{\y}\dt\e}^2 + \frac{-\sigma_0'f'(\b_0)}{2\c_0}\int_{\Omega}J\abs{\dt\b}^2 + \frac{-\sigma_0'}{2\c_0}\int_{\Sigma}\abs{\dt\c}^2\bigg)\\
&&+\bigg(\half\int_{\Omega}J\abs{\dl\dt u}^2+\frac{-\beta\sigma'_0f'(\b_0)}{\c_0}\int_{\Omega}J\abs{\nl\dt\b}^2+\frac{-\gamma\sigma'_0}{\c_0}\int_{\Sigma}\abs{\na_{\y}\dt\c}^2\bigg)\no\\
&&+\frac{-\beta\sigma'_0}{\c_0}\int_{\Sigma}\Big(\dt\c-f'(\b_0)\dt\b\Big)(\omega_{\c0}\dt\c+\omega_{\b0}\dt\b)\no\\
&=&\int_{\Omega}J\dt u\cdot F^1+\int_{\Omega}J\dt pF^2-\int_{\Sigma}\dt u\cdot F^3+\int_{\Sigma}(\dt\e-\sigma_0\de_{\y}\dt\e)F^4\no\\
&&+\frac{-\sigma'_0}{\c_0}\int_{\Sigma}\dt\c F^5+\frac{-\sigma'_0f'(\b_0)}{\c_0}\int_{\Omega}J\dt\b F^6
+\frac{\beta\sigma'_0}{\c_0}\int_{\Sigma}\Big(\dt\c+f'(\b_0)\dt\b\Big)F^7.\no
\end{eqnarray}
For the first and fourth terms inside the time derivative, we write
\begin{eqnarray}
 \half\int_{\Omega}J\abs{\dt u}^2 + \frac{-\sigma_0'f'(\b_0)}{2\c_0}\int_{\Omega}J\abs{\dt\b}^2 =  \half\int_{\Omega}\abs{\dt u}^2 + \frac{-\sigma_0'f'(\b_0)}{2\c_0}\int_{\Omega} \abs{\dt\b}^2 \no \\
 +   \half\int_{\Omega}(J-1)\abs{\dt u}^2 +\frac{-\sigma_0'f'(\b_0)}{2\c_0}\int_{\Omega}(J-1)\abs{\dt\b}^2
\end{eqnarray}
in order to eliminate the $J$ factors and absorb the error terms with $(J-1)$ into $\mathcal{F}$. We may estimate the right-hand side of \eqref{energy theorerm temporal} as
\begin{eqnarray}
RHS&\ls&\tm{\dt u}{0}\tm{JF^1}{0}+\frac{\ud{}}{\ud{t}}\int_{\Omega}JpF^2+\abs{\int_{\Sigma}p(\dt JF^2+J\dt F^2)}+\tms{\dt u}{0}\tms{F^3}{0}\\
&&+\tms{\dt\e-\sigma_0\de_{\y}\dt\e}{0}\tms{F^4}{0}+\abs{\int_{\Sigma}\dt\c F^5}+\tm{\dt\b}{0}\tm{JF^6}{0}\no\\
&&+\tms{\dt\c+f'(\b_0)\dt\b}{0}\tms{F^7}{0}.\no
\end{eqnarray}
Combining this estimate with Lemmas \ref{nonlinear lemma 1} and \ref{nonlinear lemma 2} then shows that
\begin{eqnarray}\label{energy theorerm temporal final}
&&\frac{\ud{}}{\ud{t}}\bigg(\half\int_{\Omega}\abs{\dt u}^2+\half\int_{\Sigma}\abs{\dt\e}^2
+\frac{\sigma_0}{2}\int_{\Sigma}\abs{\na_{\y}\dt\e}^2 + \frac{-\sigma_0'f'(\b_0)}{2\c_0}\int_{\Omega}\abs{\dt\b}^2 + \frac{-\sigma_0'}{2\c_0}\int_{\Sigma}\abs{\dt\c}^2 + \mathcal{F}\bigg) \\
&&+\bigg(\half\int_{\Omega}\abs{\dd \dt u}^2+\frac{-\beta\sigma'_0f'(\b_0)}{\c_0}\int_{\Omega}\abs{\na\dt\b}^2+\frac{-\gamma\sigma'_0}{\c_0}\int_{\Sigma}\abs{\na_{\y}\dt\c}^2\bigg)\no\\
&&+\frac{-\beta\sigma'_0}{\c_0}\int_{\Sigma}\Big(\dt\c-f'(\b_0)\dt\b\Big)(\omega_{\c0}\dt\c+\omega_{\b0}\dt\b)\no
\ls \sqrt{\en}\di
\end{eqnarray}

\ \\
Step 2: Spatial derivative - second order.\\
We apply $\p^{\al}$ with $\al_0=0$ and $\abs{\al}=2$ to system \eqref{perturbation}. Then we have
\begin{eqnarray}
&&\frac{\ud{}}{\ud{t}}\bigg(\half\int_{\Omega}\abs{\p^{\al}u}^2+\half\int_{\Sigma}\abs{\p^{\al}\e}^2
+\frac{\sigma_0}{2}\int_{\Sigma}\abs{\na_{\y}\p^{\al}\e}^2+\frac{-\sigma_0'f'(\b_0)}{2\c_0}\int_{\Omega}\abs{\p^{\al}\b}^2
+\frac{-\sigma_0'}{2\c_0}\int_{\Sigma}\abs{\p^{\al}\c}^2\bigg)\\
&&+\bigg(\half\int_{\Omega}\abs{\dd\p^{\al}u}^2+\frac{-\beta\sigma'_0f'(\b_0)}{\c_0}\int_{\Omega}\abs{\na\p^{\al}\b}^2
+\frac{-\gamma\sigma'_0}{\c_0}\int_{\Sigma}\abs{\na_{\y}\p^{\al}\c}^2\bigg)\no\\
&&+\frac{-\beta\sigma'_0}{\c_0}\int_{\Sigma}\Big(\p^{\al}\c-f'(\b_0)\p^{\al}\b\Big)(\omega_{\c0}\p^{\al}\c+\omega_{\b0}\p^{\al}\b)\no\\
&=&\int_{\Omega}\p^{\al}u\cdot \p^{\al}G^1+\int_{\Omega}(\p^{\al}p\p^{\al}G^2-\p^{\al}u\cdot\na \p^{\al}G^2)-\int_{\Sigma}\p^{\al}u\cdot \p^{\al}G^3+\int_{\Sigma}(\p^{\al}\e-\sigma_0\de_{\y}\p^{\al}\e)\p^{\al}G^4\no\\
&&+\frac{-\sigma'_0}{\c_0}\int_{\Sigma}\p^{\al}\c\p^{\al}G^5+\frac{-\sigma'_0f'(\b_0)}{\c_0}\int_{\Omega}\p^{\al}\b\p^{\al}G^6
+\frac{\beta\sigma'_0}{\c_0}\int_{\Sigma}\Big(\p^{\al}\c+f'(\b_0)\p^{\al}\b\Big)\p^{\al}G^7.\no
\end{eqnarray}
Assume $\p^{\al}=\p^{\vartheta}\p^{\varphi}$ for $\abs{\vartheta}=\abs{\varphi}=1$. We may integrate by parts on the right-hand side to obtain
\begin{eqnarray}
RHS&=&-\int_{\Omega}\p^{\al+\vartheta}u\cdot \p^{\varphi}G^1+\int_{\Omega}(\p^{\al}p\p^{\al}G^2-\p^{\al+\vartheta}u\cdot\na \p^{\varphi}G^2)-\int_{\Sigma}\p^{\al}u\cdot \p^{\al}G^3\\
&&+\int_{\Sigma}(\p^{\varphi}\e-\sigma_0\de_{\y}\p^{\varphi}\e)\p^{\al+\vartheta}G^4+\frac{-\sigma'_0}{\c_0}\int_{\Sigma}\p^{\al+\vartheta}\c\p^{\varphi}G^5
+\frac{-\sigma'_0f'(\b_0)}{\c_0}\int_{\Omega}\p^{\al+\vartheta}\b\p^{\varphi}G^6\no\\
&&+\frac{\beta\sigma'_0}{\c_0}\int_{\Sigma}\Big(\p^{\al}\c+f'(\b_0)\p^{\al}\b\Big)\p^{\al}G^7\no\\
&\ls&\tm{u}{3}\tm{G^1}{1}+\tm{p}{2}\tm{G^2}{2}+\tm{u}{3}\tm{G^2}{2}+\tms{u}{\frac{5}{2}}\tms{G^3}{\frac{3}{2}}\no\\
&&+\tms{\e}{\frac{7}{2}}\tms{G^4}{\frac{5}{2}}+\tms{\c}{3}\tms{G^5}{1}+\tm{\b}{3}\tm{G^6}{1}\no\\
&&+\Big(\tms{\c}{\frac{5}{2}}+\tm{\b}{3}\Big)\tms{G^7}{\frac{3}{2}}.\no
\end{eqnarray}
We then employ Lemma \ref{nonlinear lemma 3} to deduce that
\begin{eqnarray}\label{energy theorerm spatial final}
&&\frac{\ud{}}{\ud{t}}\bigg(\half\int_{\Omega}\abs{\p^{\al}u}^2+\half\int_{\Sigma}\abs{\p^{\al}\e}^2
+\frac{\sigma_0}{2}\int_{\Sigma}\abs{\na_{\y}\p^{\al}\e}^2+\frac{-\sigma_0'f'(\b_0)}{2\c_0}\int_{\Omega}\abs{\p^{\al}\b}^2
+\frac{-\sigma_0'}{2\c_0}\int_{\Sigma}\abs{\p^{\al}\c}^2\bigg)\\
&&+\bigg(\half\int_{\Omega}\abs{\dd\p^{\al}u}^2+\frac{-\beta\sigma'_0f'(\b_0)}{\c_0}\int_{\Omega}\abs{\na\p^{\al}\b}^2
+\frac{-\gamma\sigma'_0}{\c_0}\int_{\Sigma}\abs{\na_{\y}\p^{\al}\c}^2\bigg)\no\\
&&+\frac{-\beta\sigma'_0}{\c_0}\int_{\Sigma}\Big(\p^{\al}\c-f'(\b_0)\p^{\al}\b\Big)(\omega_{\c0}\p^{\al}\c+\omega_{\b0}\p^{\al}\b) \ls \sqrt{\en}\di\no
\end{eqnarray}
whenever $\abs{\alpha}=2$ and $\alpha_0=0$.

\ \\
Step 3: Spatial derivative - first order and zeroth order.\\
We apply $\p^{\al}$ with $\al_0=0$ and $\abs{\al}\leq 1$ to system \eqref{perturbation}. An argument simpler than the one used in Step 2 shows that \eqref{energy theorerm spatial final} holds also for this range of $\alpha$.

\ \\
Step 4: Conclusion.\\
To complete the estimate we sum \eqref{energy theorerm temporal final} with \eqref{energy theorerm spatial final}, with the latter applied for all $\abs{\alpha} \le 2$ and $\alpha_0=0$.

\end{proof}

\section{Comparison estimates}

In this section, we turn to the crucial comparison estimate between horizontal quantities and full quantities.

\subsection{Preliminaries}

The next result is the key to proving coercivity of the dissipation over the energy space.  It is a Poincar\'{e}-type inequality for pairs of functions.

\begin{lemma}\label{prelim lemma 1}
For $C\in H^1(\Sigma)$ and $B\in H^1(\Omega)$ satisfying
\begin{eqnarray}
\int_{\Sigma}C+\int_{\Omega}B=0,
\end{eqnarray}
we have
\begin{eqnarray}
\tms{C}{0}+\tm{B}{0}\ls \tms{\na_{\y} C}{0}+\tm{\na B}{0}+\tms{C-f'(\b_0)B}{0}.
\end{eqnarray}
\end{lemma}
\begin{proof}
Assume this claim is not true. Then we can find $\{C_n,B_n\}_{n=1}^{\infty}$ satisfying
\begin{eqnarray}
\int_{\Sigma}C_n+\int_{\Omega}B_n=0,
\end{eqnarray}
 such that
\begin{eqnarray}
\tms{C_n}{0}+\tm{B_n}{0}=1,
\end{eqnarray}
and
\begin{eqnarray}
\tms{\na_{\y} C_n}{0}+\tm{\na B_n}{0}+\tms{C_n-f'(\b_0)B_n}{0}\leq \frac{1}{n}.
\end{eqnarray}
This implies
\begin{eqnarray}
\tms{C_n}{1}\leq 2,\ \ \tm{B_n}{1}\leq 2.
\end{eqnarray}
Thus, we can extract weakly convergent subsequence
\begin{eqnarray}
C_n\rightharpoonup C_0\ \ \text{in}\ \ H^1(\Sigma),\\
B_n\rightharpoonup B_0\ \ \text{in}\ \ H^1(\Omega).
\end{eqnarray}
Also, by compactly embedding theorem and trace theorem, we know
\begin{eqnarray}
C_n\rt C_0\ \ \text{in}\ \ H^0(\Sigma),\\
B_n\rt B_0\ \ \text{in}\ \ H^0(\Omega),\\
B_n\rt B_0\ \ \text{in}\ \ H^0(\Sigma).
\end{eqnarray}
This implies
\begin{eqnarray}
\int_{\Sigma}C_0+\int_{\Omega}B_0=0,
\end{eqnarray}
\begin{eqnarray}\label{ct 02}
\tms{C_0}{0}+\tm{B_0}{0}=1,
\end{eqnarray}
and by weak lower semi-continuity,
\begin{eqnarray}\label{ct 03}
\tms{\na_{\y} C_0}{0}+\tm{\na B_0}{0}+\tms{C_0-f'(\b_0)B_0}{0}\leq 0.
\end{eqnarray}
Naturally, \eqref{ct 03} yields that $C_0$ and $B_0$ are constants satisfying
\begin{eqnarray}
C_0-f'(\b_0)B_0=0,\ \ \abs{\Sigma}_2 C_0+\abs{\Omega}_3 B_0=0
\end{eqnarray}
Since $f'(\b_0)>0$, the only solution to this linear system is $C_0=B_0=0$.
This contradicts the normalization condition \eqref{ct 02}.
Therefore, the claim is verified.
\end{proof}

\begin{lemma}\label{prelim lemma 2}
Assume $\en\leq\delta<<1$. We have
\begin{eqnarray}
\abs{\int_{\Sigma}\c}+\abs{\int_{\Omega}\b}&\ls& \sqrt{\bdi}+\sqrt{\en}\sqrt{\di}.
\end{eqnarray}
\end{lemma}
\begin{proof}
We decompose
\begin{eqnarray}
\int_{\Sigma}\c&=&\int_{\Sigma}\Big((\c_0+\c)-\c_0\Big)\\
&=&\bigg(\int_{\Sigma}(\c_0+\c)\sqrt{1+\abs{\na_{\y}\e}^2}-\abs{\Sigma}\c_0\bigg)+\int_{\Sigma}(\c_0+\c)\Big(1-\sqrt{1+\abs{\na_{\y}\e}^2}\Big),\no\\
\int_{\Omega}\b&=&\int_{\Omega}\Big((\b_0+\b)-\b_0\Big)\\
&=&\bigg(\int_{\Omega}J(\b_0+\b)-\abs{\Omega}_3 \b_0\bigg)+\int_{\Omega}(\b_0+\b)(1-J),\no
\end{eqnarray}
which, using the conservation of surfactant mass \eqref{mass}, implies
\begin{eqnarray}
\int_{\Sigma}\c+\int_{\Omega}\b&=&\bigg(\int_{\Sigma}(\c_0+\c)\sqrt{1+\abs{\na_{\y}\e}^2}+\int_{\Omega}J(\b_0+\b)-\abs{\Sigma}_2\c_0-\abs{\Omega}_3\b_0\bigg)\\
&&+\int_{\Sigma}(\c_0+\c)\Big(1-\sqrt{1+\abs{\na_{\y}\e}^2}\Big)+\int_{\Omega}(\b_0+\b)(1-J)\no\\
&=&\int_{\Sigma}(\c_0+\c)\Big(1-\sqrt{1+\abs{\na_{\y}\e}^2}\Big)+\int_{\Omega}(\b_0+\b)(1-J).\no
\end{eqnarray}
Let $K_{\c}[\e]=\ds\int_{\Sigma}(\c_0+\c)\Big(1-\sqrt{1+\abs{\na_{\y}\e}^2}\Big)$ and $K_{\b}[\e]=\ds\int_{\Omega}(\b_0+\b)(1-J)$. Since for $\en\leq\delta<<1$, Lemma \ref{nonlinear lemma 1} guarantees that
\begin{eqnarray}
\frac{\c_0}{2}\leq\c_0+\c\leq\frac{3\c_0}{2},\\
\frac{\b_0}{2}\leq\b_0+\b\leq\frac{3\b_0}{2}
\end{eqnarray}
for $\c_0$ and $\b_0$ the positive equilibrium constants, we have
\begin{eqnarray}
\abs{K_{\c}[\e]}&\ls& \tms{\na_{\y}\e}{0}^2\ls\tms{\e}{1}^2\ls \sqrt{\bdi}+\sqrt{\en}\sqrt{\di},\\
\abs{K_{\b}[\e]}&\ls& \tm{\be}{1}\ls\tms{\e}{\frac{1}{2}}\ls \sqrt{\bdi}+\sqrt{\en}\sqrt{\di}.
\end{eqnarray}
Define
\begin{eqnarray}
C=\c-\frac{K_{\c}[\e]}{\abs{\Sigma}_2},\ \ B=\b-\frac{K_{\b}[\e]}{\abs{\Omega}_3},
\end{eqnarray}
which yields
\begin{eqnarray}\label{ct 01}
\int_{\Sigma}C+\int_{\Omega}B=0.
\end{eqnarray}
Based on Lemma \ref{prelim lemma 1}, we have
\begin{eqnarray}
\tms{C}{0}+\tm{B}{0}\ls \tms{\na_{\y} C}{0}+\tm{\na B}{0}+\tms{C-f'(\b_0)B}{0}.
\end{eqnarray}
Then by the triangle inequality we know
\begin{eqnarray}
\tms{\c}{0}+\tm{\b}{0}&\ls& \tms{\na_{\y}\c}{0}+\tm{\na\b}{0}+\tms{\c-f'(\b_0)\b}{0}\\
&&+\tms{\frac{K_{\c}[\e]}{\abs{\Sigma}_2 }}{0}+\tm{\frac{K_{\b}[\e]}{\abs{\Omega}_3}}{0}+\tms{\frac{K_{\c}[\e]}{\abs{\Sigma}_2}-\frac{f'(\b_0)K_{\b}[\e]}{\abs{\Omega}_3}}{0}\no\\
&\ls&\sqrt{\bdi}+\abs{K_{\c}[\e]}+\abs{K_{\b}[\e]}\no,\\
&\ls&\sqrt{\bdi}+\sqrt{\en}\sqrt{\di}.\no
\end{eqnarray}
By Cauchy's inequality, our desired result naturally follows.
\end{proof}

\subsection{Energy comparison estimates}

Now we use various auxiliary estimates to control the full energy functional in terms of the horizontal energy.

\begin{theorem}\label{comparison theorem 1}
Suppose $\en\leq\delta<<1$. We have
\begin{eqnarray}
\en\ls\ben+\en^2.
\end{eqnarray}
\end{theorem}
\begin{proof}
It suffices to prove
\begin{eqnarray}
\tm{u}{2}^2+\tm{p}{1}^2+\tms{\dt\e}{\frac{3}{2}}^2+\tms{\dt^2\e}{-\frac{1}{2}}^2+\tm{\b}{2}^2\ls\ben+\en^2.
\end{eqnarray}
The previous four terms can be estimated as \cite[Theorem 5.2]{Kim.Tice2016}, so we focus on the last term.
Now, according to \eqref{perturbation}, we have
\begin{eqnarray}\label{elliptic.}
\left\{
\begin{array}{ll}
-\beta\de\b=-\dt\b+G^6&\ \ \text{in}\ \ \Omega,\\
\p_3\b=\omega_{\c0}\c+\omega_{\b0}\b+G^7&\ \ \text{on}\ \ \Sigma,\\
\p_3\b=0&\ \ \text{on}\ \ \Sigma_b,
\end{array}
\right.
\end{eqnarray}
Based on elliptic estimate in Lemma \ref{elliptic}, we have
\begin{eqnarray}
\tm{b-\dfrac{1}{\abs{\Omega}_3}\int_{\Omega}b}{1}
&\ls&\tm{\dt\b}{0}+\tm{G^6}{0}+\tms{\dt\c+\c_0\na_{\y}\cdot u-\de_{\y}\c-G^5}{0}+\tms{G^7}{0}\\
&\ls&\tm{\dt\b}{0}+\tm{G^6}{0}+\tms{\dt\c}{0}+\tms{u}{1}+\tms{\c}{2}\no\\
&&+\tms{G^5}{0}+\tms{G^7}{0}\no\\
&\ls&\sqrt{\ben}+\en.\no
\end{eqnarray}
Cauchy's inequality implies
\begin{eqnarray}
\abs{\int_{\Omega}b}\ls \tm{b}{0}\ls\sqrt{\ben},
\end{eqnarray}
which further yields
\begin{eqnarray}
\tm{\b}{1}&\ls&\abs{\int_{\Omega}b}+\tm{b-\dfrac{1}{\abs{\Omega}_3}\int_{\Omega}b}{1}\ls\sqrt{\ben}+\en\ls \sqrt{\ben}+\en.
\end{eqnarray}
Further, we have
\begin{eqnarray}
\tm{b-\dfrac{1}{\abs{\Omega}_3}\ds\int_{\Omega}b}{2}&\ls&\tm{\dt\b}{0}+\tm{G^6}{0}+\tms{\omega_{\c0}\c+\omega_{\b0}\b+G^7}{\frac{1}{2}}\\
&\ls&\tm{\dt\b}{0}+\tm{G^6}{0}+\tms{\c}{\frac{1}{2}}+\tm{\b}{1}+\tms{G^7}{\frac{1}{2}}\no\\
&\ls&\sqrt{\ben}+\en.\no
\end{eqnarray}
Hence, we know
\begin{eqnarray}
\tm{\b}{2}\ls \abs{\int_{\Omega}b}+\sqrt{\ben}+\en\ls \sqrt{\ben}+\en.
\end{eqnarray}
\end{proof}

\subsection{Dissipation comparison estimates}

Next we control the full dissipation in terms of the horizontal dissipation.

\begin{theorem}\label{comparison theorem 2}
Suppose $\en\leq\delta<<1$. We have
\begin{eqnarray}
\di\ls\bdi+\en\di.
\end{eqnarray}
\end{theorem}
\begin{proof}
In the full dissipation
\begin{eqnarray}
\di&=&\tm{u}{3}^2+\tm{\dt u}{1}^2+\tm{p}{2}^2+\tms{\e}{\frac{7}{2}}^2+\tms{\dt\e}{\frac{5}{2}}^2+\tms{\dt^2\e}{\frac{1}{2}}^2\\
&&+\tm{\b}{3}^2+\tm{\dt\b}{1}^2+\tms{\c}{3}^2+\tms{\dt\c}{1}^2,\no
\end{eqnarray}
the estimates of $(u,p,\e)$ can be proved as \cite[Theorem 5.3]{Kim.Tice2016}, so
it suffices to show
\begin{eqnarray}
\tms{\c}{3}+\tms{\dt\c}{1}+\tm{\b}{3}^2+\tm{\dt\b}{1}^2\ls\sqrt{\bdi}+\sqrt{\en}\sqrt{\di}.
\end{eqnarray}
Based on Lemma \ref{prelim lemma 2} and Poincar\'e-Wirtinger inequality, we have
\begin{eqnarray}
\tms{\c}{3}\ls \tms{\na_{\y}\c}{2}+\abs{\int_{\Sigma}\c}\ls\sqrt{\bdi}+\sqrt{\en}\sqrt{\di}.
\end{eqnarray}
Also, we know
\begin{eqnarray}
\abs{\int_{\Sigma}\dt\c}&\ls&\tms{G^5}{0}+\tms{\omega_{\c0}\c+\omega_{\b0}\b}{0}\ls \sqrt{\bdi}+\sqrt{\en}\sqrt{\di},
\end{eqnarray}
since $\omega_{\c0}\c+\omega_{\b0}\b=C_0\Big(\c-f'(\b_0)\b\Big)$ for some constant $C_0>0$. Hence, by Poincar\'e-Wirtinger inequality,
\begin{eqnarray}
\tms{\dt\c}{1}\ls \tms{\na_{\y}\dt\c}{0}+\abs{\int_{\Sigma}\dt\c}\ls\sqrt{\bdi}+\sqrt{\en}\sqrt{\di}.
\end{eqnarray}
Then we need to reconsider the elliptic system \eqref{elliptic.} and try to bound by dissipation.
Based on elliptic estimate in Lemma \ref{elliptic}, we have
\begin{eqnarray}
\tm{b-\dfrac{1}{\abs{\Omega}_3}\int_{\Omega}b}{1}&\ls&\tm{\dt\b}{0}+\tm{G^6}{0}+\tms{\omega_{\c0}\c+\omega_{\b0}\b}{0}+\tms{G^7}{0}\\
&\ls&\sqrt{\bdi}+\sqrt{\en}\sqrt{\di}.\no
\end{eqnarray}
Combining this with Lemma \ref{prelim lemma 2}, we have
\begin{eqnarray}
\tm{b}{1}\ls \abs{\int_{\Omega}b}+\tm{b-\dfrac{1}{\abs{\Omega}_3}\int_{\Omega}b}{1}\ls \sqrt{\bdi}+\sqrt{\en}\sqrt{\di}.
\end{eqnarray}
Further, we have
\begin{eqnarray}
\tm{b-\dfrac{1}{\abs{\Omega}_3}\int_{\Omega}b}{2}&\ls&\tm{\dt\b}{0}+\tm{G^6}{0}+\tms{\omega_{\c0}\c+\omega_{\b0}\b+G^7}{\frac{1}{2}}\\
&\ls&\tm{\dt\b}{0}+\tm{G^6}{0}+\tms{\c}{\frac{1}{2}}+\tm{\b}{1}+\tms{G^7}{\frac{1}{2}}\no\\
&\ls&\sqrt{\bdi}+\sqrt{\en}\sqrt{\di},\no
\end{eqnarray}
Similarly, we obtain
\begin{eqnarray}
\tm{b}{2}\ls \abs{\int_{\Omega}b}+\tm{b-\dfrac{1}{\abs{\Omega}_3}\int_{\Omega}b}{2}\ls \sqrt{\bdi}+\sqrt{\en}\sqrt{\di}.
\end{eqnarray}
On the other hand, by Cauchy's inequality, we know
\begin{eqnarray}
\abs{\int_{\Omega}\dt b}&\ls&\tm{G^6}{0}+\tms{\omega_{\c0}\c+\omega_{\b0}\b+G^7}{0}\\
&\ls&\tm{G^6}{0}+\tms{\omega_{\c0}\c+\omega_{\b0}\b}{0}+\tms{G^7}{0}\ls\sqrt{\bdi}+\sqrt{\en}\sqrt{\di}.\no
\end{eqnarray}
By the Poincar\'e-Wirtinger inequality, this implies
\begin{eqnarray}
\tm{\dt\b}{1}&\ls&\tm{\na\dt\b}{0}+\abs{\int_{\Omega}\dt b}\ls \sqrt{\bdi}+\sqrt{\en}\sqrt{\di}.
\end{eqnarray}
Therefore, we continue using elliptic estimate to obtain
\begin{eqnarray}
\tm{b-\dfrac{1}{\abs{\Omega}_3}\int_{\Omega}b}{3}&\ls&\tm{\dt\b}{1}+\tm{G^6}{1}+\tms{\omega_{\c0}\c+\omega_{\b0}\b+G^7}{\frac{3}{2}}\\
&\ls&\tm{\dt\b}{1}+\tm{G^6}{1}+\tms{\c}{\frac{3}{2}}+\tm{\b}{2}+\tms{G^7}{\frac{3}{2}}\no\\
&\ls&\sqrt{\bdi}+\sqrt{\en}\sqrt{\di}.\no
\end{eqnarray}
Naturally, we have
\begin{eqnarray}
\tm{b}{3}\ls \abs{\int_{\Omega}b}+\tm{b-\dfrac{1}{\abs{\Omega}_3}\int_{\Omega}b}{3}\ls \sqrt{\bdi}+\sqrt{\en}\sqrt{\di}.
\end{eqnarray}
In summary, we have shown the desired estimate
\begin{eqnarray}
\tm{\dt\b}{1}+\tm{b}{3}&\ls&\sqrt{\bdi}+\sqrt{\en}\sqrt{\di}.
\end{eqnarray}
\end{proof}

\section{Proof of main results}

We now synthesize the previous results to prove the main a priori estimates.

\begin{theorem}\label{apriori_est}
Suppose $(u,p,\e,\c,\b)$ are solution to \eqref{perturbation} in $t\in[0,T]$. Then there exists a universal $0<\delta<<1$ such that for any
\begin{eqnarray}
\sup_{t\in[0,T]}\en(t)\leq\delta,\ \ \int_0^T\di(t)\ud{t}<\infty,
\end{eqnarray}
we have
\begin{eqnarray}
\sup_{t\in[0,T]}\ue^{\lambda t}\en(t)+\int_0^T\di(t)\ud{t}\ls\en(0),
\end{eqnarray}
for some universal $\lambda>0$.
\end{theorem}
\begin{proof}
Theorem \ref{comparison theorem 1} and Theorem \ref{comparison theorem 2} justify
\begin{eqnarray}\label{apriori_est_1}
\ben\ls\en\ls\ben,\ \ \bdi\ls\di\ls\bdi.
\end{eqnarray}
Theorem \ref{energy theorem} justifies
\begin{eqnarray}\label{ft 01}
\frac{\ud{}}{\ud{t}}(\ben-\f)+\bdi\ls0.
\end{eqnarray}
Also, Lemmas \ref{nonlinear lemma 1} and \ref{nonlinear lemma 2} together with \eqref{apriori_est_1} show that
\begin{eqnarray}
\abs{\f}\leq \frac{\ben}{4}+C\en^{\frac{3}{2}}=\bigg(\frac{1}{4}+C\sqrt{\en}\bigg)\ben,
\end{eqnarray}
which implies that for $\delta$ smaller than a universal constant
\begin{eqnarray}
\frac{\ben}{2}\leq \ben-\f\leq \frac{3\ben}{2}.
\end{eqnarray}
Hence, integrating over $[0,T]$ in \eqref{ft 01}, we have
\begin{eqnarray}
\int_0^T\di(t)\ud{t}\ls \Big(\ben(T)-\f(T)\Big)+\int_0^T\di(t)\ud{t}\ls\ben(0)+\f(0),
\end{eqnarray}
which further implies
\begin{eqnarray}
\int_0^T\di(t)\ud{t}\ls \en(0).
\end{eqnarray}
Also, it holds naturally that
\begin{eqnarray}
\ben-\f\ls\ben\ls\di,
\end{eqnarray}
which means there exists $\lambda>0$ such that
\begin{eqnarray}
\frac{\ud{}}{\ud{t}}(\ben-\f)+\lambda(\ben-\f)\leq0.
\end{eqnarray}
Therefore, this implies
\begin{eqnarray}
\ben(t)\ls\ben(t)-\f(t) \ls \ue^{-\lambda t}\Big(\ben(0)-\f(0)\Big)\ls\ue^{-\lambda t}\en(0),
\end{eqnarray}
which yields our desired decay estimate.
\end{proof}

\appendix

\makeatletter
\renewcommand \theequation {%
A.%
\@arabic\c@equation} \@addtoreset{equation}{section}
\@addtoreset{equation}{section} \makeatother

\section{Surface Differential Operators}

The proof of following lemmas can be found in \cite[Appendix A]{Kim.Tice2016}.
\begin{lemma}\label{app lemma 1}
We have the following identities:
\begin{eqnarray}
\na_{\Gamma}\cdot\nu=\na_{\y}\nu_{\y}=-H,
\end{eqnarray}
\begin{eqnarray}
\p_i\sqrt{1+\abs{\na_{\y}\e}^2}=-\nu_{\y}\cdot\na_{\y}\p_i\e,
\end{eqnarray}
\begin{eqnarray}
\dt\sqrt{1+\abs{\na_{\y}\e}^2}=\na_{\y}\left(\frac{\dt\e\na_{\y}\e}{\sqrt{1+\abs{\na_{\y}\e}^2}}\right)-\dt\e H,
\end{eqnarray}
\begin{eqnarray}
\na_{\Gamma}f\cdot\nu=0.
\end{eqnarray}
\end{lemma}

\begin{lemma}\label{app lemma 2}
We have the following identities:
\begin{eqnarray}
\int_{\Sigma}\p_{\Gamma,i}fg\sqrt{1+\abs{\na_{\y}\e}^2}=-\int_{\Sigma}(f\p_{\Gamma,i}g+fg\nu_iH)\sqrt{1+\abs{\na_{\y}\e}^2},
\end{eqnarray}
\begin{eqnarray}
\int_{\Sigma}\na_{\Gamma}\cdot X\sqrt{1+\abs{\na_{\y}\e}^2}=-\int_{\Sigma}(X\cdot\nu)H\sqrt{1+\abs{\na_{\y}\e}^2}.
\end{eqnarray}
\end{lemma}

\begin{lemma}\label{app lemma 3}
Assume $f\in C^2(\r)$. Suppose $\tilde\c$ and $\e$ satisfies
\begin{eqnarray}
\left\{
\begin{array}{rcl}
\dt\tilde\c+u\cdot\na_{\y}\tilde\c+\tilde\c\na_{\Gamma}\cdot u&=&\gamma\de_{\Gamma}\tilde\c-X,\\
\dt\e-(u\cdot\nu)\sqrt{1+\abs{\na_{\y}\e}^2}&=&0.
\end{array}
\right.
\end{eqnarray}
Then we have
\begin{eqnarray}
\frac{\ud{}}{\ud{t}}\int_{\Sigma}f(\tilde\c)\sqrt{1+\abs{\na_{\y}\e}^2}&=&\int_{\Sigma}\bigg(\Big(f(\tilde\c)-f'(\tilde\c)\tilde\c\Big)\na_{\Gamma}\cdot u-f''(\tilde\c)\abs{\na_{\Gamma}\tilde\c}^2-f'(\tilde\c)X\bigg)\sqrt{1+\abs{\na_{\y}\e}^2}.
\end{eqnarray}
\end{lemma}

\makeatletter
\renewcommand \theequation {%
B.%
\@arabic\c@equation} \@addtoreset{equation}{section}
\@addtoreset{equation}{section} \makeatother

\section{Analytic Tools}

Define Poisson integral from $\Sigma$ to $\Sigma\times(-\infty,0]$ as
\begin{eqnarray}\label{Poisson}
\pp[f](x)=\sum_{n\in(L_1^{-1}\mathbb{Z})\times(L_2^{-1}\mathbb{Z})}\ue^{2\pi\ui n x_{\y}}\ue^{2\pi\abs{n}x_3}\hat f(n),
\end{eqnarray}
where
\begin{eqnarray}
\hat f(n)=\int_{\Sigma}f(x_{\y})\frac{\ue^{2\pi\ui nx_{\y}}}{L_1L_2}\ud{x_{\y}}.
\end{eqnarray}
$\pp$ is a bounded linear operator from $H^s(\Sigma)$ to $H^{s+\frac{1}{2}}(\Sigma\times(-\infty,0])$ for $s>0$.

The proof of following lemmas can be found in \cite[Appendix B]{Kim.Tice2016}.
\begin{lemma}
We have for $q\geq1$,
\begin{eqnarray}
\tm{\na^q\pp[f]}{0}\ls\nm{f}_{\dot{H}^{q-\frac{1}{2}}(\Sigma)}.
\end{eqnarray}
\end{lemma}

\begin{lemma}
We have for $q\geq1$ and $s>1$,
\begin{eqnarray}
\nm{\na^q\pp[f]}_{L^{\infty}(\Omega)}\ls\nm{f}_{\dot{H}^{q+s}(\Sigma)}.
\end{eqnarray}
The same estimate holds for $q=0$ if $\ds\int_{\Sigma}f=0$.
\end{lemma}

\begin{lemma}
Let $U$ denote either $\Sigma$ or $\Omega$.
\begin{enumerate}
\item
Let $0\leq r\leq s_1\leq s_2$ be such that $s_1>n/2$. Let $f\in
H^{s_1}(U)$, $g\in H^{s_2}(U)$. Then $fg\in H^r(U)$ and
\begin{eqnarray}\label{ap 01}
\nm{fg}_{H^r(U)}\ls \nm{f}_{H^{s_1}(U)}\nm{g}_{H^{s_2}(U)}.
\end{eqnarray}
\item
Let $0\leq r\leq s_1\leq s_2$ be such that $s_2>r+n/2$. Let $f\in
H^{s_1}(U)$, $g\in H^{s_2}(U)$. Then $fg\in H^r(U)$ and
\begin{eqnarray}\label{ap 02}
\nm{fg}_{H^r(U)}\ls \nm{f}_{H^{s_1}(U)}\nm{g}_{H^{s_2}(U)}.
\end{eqnarray}
\item
Let $0\leq r\leq s_1\leq s_2$ be such that $s_2>r+n/2$. Let
$f\in H^{-s}(U)$, $g\in H^{s_2}(U)$. Then $fg\in H^{-s_1}(U)$ and
\begin{eqnarray}
\nm{fg}_{H^{-s_1}(U)}\ls \nm{f}_{H^{-r}(U)}\nm{g}_{H^{s_2}(U)}.
\end{eqnarray}
\end{enumerate}
\end{lemma}

In the comparison estimate, we need the following classical elliptic estimate.
\begin{lemma}\label{elliptic}
Assume $\phi\in H^{r-2}(\Omega)$ and $\psi\in H^{r-3/2}(\Sigma)$ for $r\geq1$. Then solutions to the elliptic problem
\begin{eqnarray}
\left\{
\begin{array}{ll}
-\de\b=\phi&\ \ \text{in}\ \ \Omega,\\
\p_3\b=\psi&\ \ \text{on}\ \ \Sigma,\\
\p_3\b=0&\ \ \text{on}\ \ \Sigma_b,
\end{array}
\right.
\end{eqnarray}
satisfy the estimate
\begin{eqnarray}
\tm{b-\dfrac{1}{\abs{\Omega}_3}\ds\int_{\Omega}b}{r}\ls \tm{\phi}{r-2}+\tms{\psi}{r-\frac{3}{2}}.
\end{eqnarray}
\end{lemma}

\end{document}